\documentclass[a4paper, 11pt]{article}

\usepackage{amssymb, amsmath, bbm, amsthm}
\usepackage{fullpage}
\usepackage{graphicx}
\usepackage{url}
\theoremstyle{plain}
\newtheorem{theorem}{Theorem}[section]
\newtheorem{lemma}[theorem]{Lemma}
\newtheorem{corollary}[theorem]{Corollary}
\newtheorem{proposition}[theorem]{Proposition}
\newtheorem{example}[theorem]{Example}
\numberwithin{equation}{section}
\theoremstyle{definition}
\newtheorem{definition}[theorem]{Definition}
\newtheorem{assumption}[theorem]{Assumption}
\theoremstyle{remark}
\newtheorem{remark}[theorem]{Remark}
\newcommand{\R}{{\mathbb R}}
\newcommand{\Rex}{{\overline{\mathbb R}}}
\newcommand{\N}{{\mathbb N}}
\newcommand{\F}{{\mathbb F}}
\newcommand{\G}{{\mathcal G}}
\newcommand{\thstar}{\mathop{\theta^*}}
\newcommand{\zstar}{\mathop{z^*}}

\newcommand{\xstar}{\mathop{x^*}}
\newcommand{\xminus}{\mathop{x_-}}
\newcommand{\xR}{\mathop{x_R}}
\newcommand{\xplus}{\mathop{x_+}}
\newcommand{\e}{\epsilon}
\newcommand{\Prob}{\mathop{\mathbb{P}}\nolimits}
\newcommand{\E}{\mathop{\mathbb{E}}\nolimits}

\newcommand{\sX}{{\mathcal X}}
\newcommand{\hth}{\hat{\theta}}
\newcommand{\hx}{\hat{x}}
\newcommand{\bart}{\bar{\theta}}
\newcommand{\Keywords}[1]{\par\noindent 
{\small{\em Keywords\/}: #1}}

\makeatletter
\def\Ddots{\mathinner{\mkern1mu\raise\p@
\vbox{\kern7\p@\hbox{.}}\mkern2mu
\raise4\p@\hbox{.}\mkern2mu\raise7\p@\hbox{.}\mkern1mu}}
\makeatother
\title{Constructing Time-Homogeneous Generalised Diffusions 
Consistent with Optimal Stopping Values}

\author{David Hobson\thanks{Corresponding author, D.Hobson@warwick.ac.uk} \ 
and Martin Klimmek\thanks{M.Klimmek@warwick.ac.uk} \\
Department of Statistics, University of Warwick}

\date{\today}
\begin{document}
\maketitle

\begin{abstract}
Consider a set of discounted optimal stopping problems for a 
one-parameter family of objective functions and a fixed diffusion 
process, started at a fixed point. A standard problem in stochastic 
control/optimal stopping is to solve for the problem value in this 
setting.

In this article we consider an inverse problem; given the set of problem 
values for a family of objective functions, we aim to recover the 
diffusion. Under a natural assumption on the family of objective 
functions we can characterise existence and uniqueness of a diffusion 
for which the optimal stopping problems have the specified values. The 
solution of the problem relies on techniques from generalised convexity theory. \smallskip

\Keywords{{\it optimal stopping, generalised convexity, generalised 
diffusions, inverse American option problem}}
\end{abstract}

\section{Introduction}
Consider a classical optimal stopping problem in which we are
given a discount parameter, an objective function 
and a time-homogeneous diffusion process started at a fixed point, and
we are asked to maximise the expected discounted payoff. Here the payoff 
is the objective function evaluated at the value of the diffusion at a 
suitably chosen stopping time. We call this problem the 
forward optimal stopping problem, and the expected payoff under the 
optimal stopping rule the (forward) problem value.
 
The set-up can be generalised to a one-parameter family of objective 
functions to give a one-parameter family of problem values. In this 
article we are interested in an associated inverse problem. The inverse 
problem is, given a one-parameter family of objective functions and 
associated optimal values, to recover the underlying diffusion, or 
family of diffusions, for which the family of forward stopping problems 
yield the given values.

The approach of this article is to exploit the structure of the optimal 
control problem and the theory of generalised convexity from convex 
analysis to obtain a duality relation between the Laplace transform of 
the first hitting time and the set of problem values. The Laplace 
transform can then be inverted to give the diffusion process.

The generalised convexity approach sets this article apart from previous 
work on this problem, see \cite{alfonsi3, alfonsi2, hobson}. All these 
papers are set in the realm of mathematical finance where 
the values of the stopping problems can be identified with the 
prices of perpetual American options, and the diffusion process is the 
underlying stock process. In that context, it is a natural question to ask: Given a 
set of perpetual American option prices from the market, parameterised 
by the strike, is it possible to identify a model consistent with all 
those prices simultaneously? In this article we abstract from the 
finance setting and ask a more general question: When can we identify a 
time-homogeneous diffusion for which the values of a parameterised 
family of optimal stopping problems coincide with a pre-specified function 
of the parameter.

Under restrictive smoothness assumptions on the volatility 
coefficients, Alfonsi and Jourdain~\cite{alfonsi3}  
develop a `put-call 
parity' which relates the prices of perpetual American puts (as a 
function of strike) under one model to the prices of perpetual American 
calls (as a function of the initial value of the underlying asset) under 
another model.
This correspondence 
is extended to other payoffs in \cite{alfonsi2}. The result is then 
applied
to solve the inverse problem described above. In both papers the idea is 
to find a coupled pair of free-boundary problems, the solutions of which 
can be used to give a relationship between the pair of model 
volatilities. 

In contrast, in Ekstr\"{o}m and Hobson \cite{hobson} the idea is to 
solve the inverse problem by exploiting a duality between the 
put price 
and the Laplace transform of the first hitting time. This duality gives 
a direct approach to the inverse problem. It is based on a convex 
duality which requires no smoothness on the volatilities or option 
prices. 

In this article we consider a general inverse problem of how 
to recover 
a diffusion which is consistent with a given set of values for a family 
of optimal stopping problems. The solution requires the use of 
generalised, or $u$-convexity (Carlier~\cite{carlier}, 
Villani~\cite{villani}, Rachev and R\"{u}schendorf~\cite{rachev}). 
The log-value function is the $u$-convex dual
of the log-eigenfunction of the generator (and vice-versa) 
and the $u$-subdifferential corresponds to the 
optimal stopping threshold. These simple concepts give a direct 
and probabilistic approach 
to the inverse problem which contrasts with the involved calculations 
in \cite{alfonsi3,alfonsi2} in which pdes play a key role.

A major advantage of the dual approach is that there are no smoothness 
conditions on the value function or on the diffusion. In particular, it 
is convenient to work with generalised diffusions which are specified by 
the speed measure (which may have atoms, and intervals which have zero 
mass).

{\bf Acknowledgement:} DGH would like to thank Nizar Touzi for 
suggesting generalised convexity as an approach for this problem. 

\section{The Forward and the Inverse Problems}
\label{s:tfatip}

Let $\sX$ be a class of diffusion processes, let $\rho$ be a 
discount parameter, and let $\G=\{G(x,\theta); \theta \in \Theta \}$ be a 
family of non-negative 
objective functions, parameterised by a real parameter $\theta$ which 
lies in an interval $\Theta$.
The forward problem, which is standard in optimal stopping, is for a 
given $X \in \sX$, to 
calculate for each $\theta \in \Theta$, the problem value
\begin{equation} \label{eq:forward}
V(\theta) \equiv V_X(\theta) =\sup_\tau \E_0[e^{-\rho \tau} 
G(X_\tau,\theta)] ,
\end{equation}
where the supremum is taken over finite stopping times $\tau$, and 
$\E_0$ denotes the fact that $X_0=0$.
The inverse problem is, given a fixed $\rho$ and the family $\G$, to 
determine whether 
$V \equiv \{ V(\theta) :  \theta \in \Theta \}$ could have arisen as a
solution to the family of 
problems 
(\ref{eq:forward}) and if so, to
characterise those elements $X \in \sX$ which would lead to the value 
function $V$. The inverse problem, which is the main object of our 
analysis, is much less standard than the forward problem, but has 
recently been the subject of 
some studies (\cite{alfonsi3, alfonsi2, hobson}) in the context of 
perpetual American options. In these papers the space of candidate 
diffusions is 
$\sX_{stock}$, where $\sX_{stock}$ is the set of price processes which, 
when 
discounted, are martingales and $G(x,\theta) = (\theta-x)^+$ is the put 
option 
payoff (slightly more general payoffs are considered in 
\cite{alfonsi2}). The aim 
is to find a stochastic model which is 
consistent 
with an observed continuum of perpetual put prices.

In fact it will be convenient in this article to extend the set $\sX$ to 
include the set of generalised diffusions in the sense of It\^{o} and 
McKean~\cite{mckean}. These diffusions are generalised in the sense that 
the speed measure may include atoms, or regions with zero or infinite 
mass. Generalised diffusions can be constructed as time changes of 
Brownian Motion, see Section~\ref{SS:existence} below, and also \cite{mckean}, 
\cite{watanabe}, \cite{rogers}, and 
for a setup related to the one considered here, \cite{hobson}.

We will concentrate on the set of generalised diffusions started and 
reflected at $0$, which are local martingales (at least when away from 
zero). We denote this class $\sX_0$. (Alternatively we can think of 
an element $X$ as the modulus of a local martingale $Y$ whose 
characteristics are symmetric about the initial point zero.) The twin 
reasons for focusing on $\sX_0$ rather than $\sX$, are that the optimal 
stopping problem is guaranteed to become one-sided rather than 
two-sided, and 
that 
within $\sX_0$ there is some hope of finding a unique solution to the 
inverse 
problem. The former reason is more fundamental (we will comment in 
Section~\ref{ss:otherX} below 
on other plausible choices of subsets of $\sX$ for which a similar 
approach is equally fruitful). For $X \in \sX_0$, 0 is a reflecting 
boundary and we assume a natural right boundary but we do not exclude 
the possibility that it is absorbing.
Away from zero the process is in natural scale and can be characterised 
by its speed measure, and in the case of a classical diffusion by the 
diffusion coefficient $\sigma$. In that case we may consider $X \in 
\sX_0$ to be a solution of the SDE (with reflection)
\[ dX_t = \sigma(X_t) dB_t + dL_t \hspace{20mm} X_0=0, \]
where $L$ is the local time at zero.

We return to the (forward) optimal stopping problem: 
For fixed $X$ define $\varphi(x) = \varphi_X(x) = \E_0[e^{- \rho H_x 
}]^{-1}$, where $H_x$ is the first hitting time of level $x$. Let
\begin{equation} \label{hatv}
\hat{V}(\theta)=\sup_{x : \varphi(x)<\infty} \left[ G(x,\theta) 
\E_0[e^{-\rho H_x}] 
\right] 
= \sup_{x : \varphi(x)<\infty} \left[ \frac{G(x,\theta)}{\varphi(x)} 
\right]. 
\end{equation}
Clearly $V \geq \hat{V}$. Indeed, as the following lemma shows, there is 
equality and for the forward problem 
(\ref{eq:forward}), the search over all 
stopping times can be reduced to a search over first hitting times.

\begin{lemma} \label{l:coincide}
$V$ and $\hat{V}$ coincide.
\end{lemma}

\begin{proof} See Appendix.
\end{proof}

The first step in our approach will be to take logarithms which converts 
a multiplicative problem into an additive one.
Introduce the notation
\begin{eqnarray*}
v(\theta) &=& \log(V(\theta)), \\
g(x,\theta) &=& \log(G(x,\theta)), \\
\psi(x) &=& \log(\E_0[e^{-\rho H_x}]^{-1}) = \log \varphi(x).
\end{eqnarray*}
Then the equivalent $\log$-transformed problem (compare 
(\ref{hatv})) is 
\begin{equation}
v(\theta)=\sup_{x}[g(x,\theta)-\psi(x)] ,
\label{eq:logvupsi}
\end{equation}
where the supremum is taken over those $x$ for which $\psi(x)$ is 
finite. To each of these quantities we may attach the superscript $X$ 
if we wish to associate the solution of the forward problem to a 
particular diffusion. For reasons which will become apparent, see 
Equation (\ref{eq:differential}) below, we call $\varphi_X$ the eigenfunction (and 
$\psi_X$ 
the log-eigenfunction) associated with $X$.

In the case where $g(x,\theta)=\theta x$, $v$ and $\psi$ are convex 
duals. More generally the relationship between $v$ and $\psi$ is that
of $u$-convexity (\cite{carlier}, \cite{villani}, \cite{rachev}). 
(In Section~\ref{S:uconvex} we give the 
definition of the $u$-convex dual $f^u$ of a function $f$, and derive 
those properties that we will need.) For our setting, and
under mild regularity assumptions on the functions $g$, see 
Assumption~\ref{ass} below, we will show that there is a 
duality relation between $v$ 
and $\psi$ via the $\log$-payoff function $g$ which can be exploited to 
solve both the forward and inverse problems. 
In particular our main results (see Proposition~\ref{p:forward} and 
Theorems 
\ref{t:existence} and \ref{t:suffexist} for precise statements) 
include: 

\noindent{\bf Forward Problem:}
Given a diffusion $X \in \sX_0$, let $\varphi_X(x) = 
( \E_0[e^{-\rho H_x}])^{-1}$ and
$\psi_X(x)=\log(\varphi_X(x))$. 
Set $\psi^g(\theta) = \sup_x \{ g(x,\theta)-\psi(x) \}$.
Then 
the solution to the forward problem is given by 
$V(\theta)=\exp(\psi^g(\theta))$, at least for those $\theta$ for which 
there is an optimal, finite stopping rule.
We also find that $V$ is locally Lipschitz over the same range of 
$\theta$.

\noindent{\bf Inverse Problem:} For $v=\{ v(\theta): \theta \in \Theta 
=[\theta_-,\theta_+]
\}$ to be logarithm of the solution of (\ref{eq:forward}) for some $X 
\in \sX_0$ it is sufficient   
that the $g$-convex dual (given  by $v^g(x)=\sup_{\theta} \{ 
g(x,\theta)-v(\theta)\}$) satisfies $v^g(0)=0$, 
$e^{v^g(x)}$ is convex and increasing, and $v^g(x) > \{ g(x,\theta_-) - 
g(0,\theta_-)\}$ for all $x>0$.

Note that in stating the result for the inverse problem we have assumed 
that $\Theta$ contains its endpoints, but this is not necessary, and 
our theory will allow for $\Theta$ to be open and/or unbounded at 
either end.

If $X$ is a solution of the inverse problem then we will say that $X$ is 
consistent with $\{ V(\theta); \theta \in \Theta \}$. By abuse of 
notation we will say that $\varphi_X$ (or $\psi_X$) is consistent with 
$V$ (or $v = \log V$) if, when solving the optimal stopping problem
(\ref{eq:forward}) for the diffusion with eigenfunction $\varphi_X$, we 
obtain the problem values $V(\theta)$ for each $\theta \in \Theta$.

The main technique in the proofs of these results is to exploit 
(\ref{eq:logvupsi}) to relate the 
fundamental solution $\varphi$ with $V$. Then there is a second part 
of the problem which is to relate $\varphi$ to an element of $\sX$. In 
the case where we restrict attention to $\sX_0$, each increasing convex 
$\varphi$ with 
$\varphi(0)=1$ is associated with a unique generalised diffusion $X \in 
\sX_0$. Other choices of subclasses of $\sX$ may or may not have this 
uniqueness property. See the discussion in Section~\ref{SS:uniqueness}.

The following examples give an idea of the scope of the problem:

\begin{example} \label{ex:sforward1}
Forward Problem:
Suppose $G(x,\theta)=e^{x\theta}$. Let $m>1$ and suppose that
$X \in \sX_0$ solves $dX = \sigma(X) dW + dL$ for $\sigma(x)^{-2}= 
(x^{2(m-1)}+(m-1)x^{m-2})/(2 \rho)$. For such a diffusion 
$\varphi(x)=\exp(\frac{1}{m} x^m), x \geq 0$. 
Then for $\theta \in \Theta = (0,\infty)$, 
$V(\theta)=\exp(\frac{m-1}{m} 
\theta^{\frac{m}{m-1}})$.
\end{example}

\begin{example}\label{ex:BM1a}
Forward Problem: Let $X$ be reflecting Brownian Motion on the positive 
half-line with a
natural boundary at $\infty$. Then $\varphi(x)=\cosh(x \sqrt{2 \rho})$.
Let $g(x,\theta)=\theta x$ so that $g$-convexity is standard convexity, 
and suppose $\Theta = (0,\infty)$. 
Then
\[v(\theta)=\sup_x[\theta x - \log(\cosh(x \sqrt{2 \rho }))] .\] 
It is
easy to ascertain that the supremum is attained at $x=\xstar(\theta)$ 
where
\begin{equation}
%\theta&=&\sqrt{2 \rho}\tanh(\xstar(\theta) \sqrt{2 \rho}) \\
\xstar(\theta) = \frac{1}{\sqrt{2 \rho}} \tanh^{-1} \left( \frac{\theta}{\sqrt{2 \rho}} \right)
\end{equation}
for $\theta \in [0,\sqrt{2 \rho})$. Hence, for $\theta \in (0,\sqrt{2 \rho})$
\begin{eqnarray*}
v(\theta)&=& \frac{\theta}{\sqrt{2 \rho}} \tanh^{-1} 
\left( \frac{\theta}{\sqrt{2 \rho}} \right) - \log \left( \cosh 
\tanh^{-1} \left( \frac{\theta}{\sqrt{2 \rho}} \right) \right) \\
&=& \frac{\theta}{\sqrt{2 \rho}} \tanh^{-1} 
\left( \frac{\theta}{\sqrt{2 \rho}} \right) 
+ \frac{1}{2} \log \left(1-\frac{\theta^2}{2\rho} \right),
\end{eqnarray*}
with limits $v(0)=0$ and $v(\sqrt{2 \rho}) = \log 2$.
For $\theta > \sqrt{2 \rho}$ we have $v(\theta)=\infty$.
\end{example}

\begin{example} Inverse Problem:
Suppose that $g(x,\theta)=\theta x$ and $\Theta = (0, \sqrt{2 \rho})$. 
Suppose also  
that for $\theta \in \Theta$
\[ V(\theta) = \exp \left( \frac{\theta}{\sqrt{2 \rho}} \tanh^{-1} 
\left( \frac{\theta}{\sqrt{2 \rho}} \right) 
+ \frac{1}{2} \log \left(1-\frac{\theta^2}{2\rho} 
\right) \right).
\]
Then $X$ is reflecting Brownian Motion.

Note that $X \in \sX_0$ is uniquely determined, and its diffusion 
coefficient is specified on $\R^+$. In particular, if we expand the 
domain of definition of $\Theta$ to $(0,\infty)$ then for consistency we 
must have $V(\Theta)=\infty$ for $\theta>\sqrt{2 \rho}$.
\end{example}

\begin{example} \label{ex:powers} 
Inverse Problem: Suppose $G(x,\theta)=x^\theta$ and 
$V(\theta)=\{\frac{\theta^{\frac{\theta}{2}} 
(2-\theta)^{\frac{2-\theta}{2}}}{2} : \theta \in (1,2) \}$. Then 
$\varphi(x)=1+x^2$ for $x>1$ and, at least whilst $X_t>1$, $X$ solves 
the SDE $dX = \rho(1+X)^2 dW$. In particular, $V$ does not contain 
enough information to determine a unique consistent diffusion in $\sX_0$ 
since there is some indeterminacy of the diffusion co-efficient on 
$(0,1)$.
\end{example}

\begin{example} \label{ex:concave}
Inverse Problem: Suppose $g(x,\theta)={-\theta^2}/(2\{1+x\})$, $\Theta 
=[1,\infty)$ and 
$v(\theta)=\{-1/2 - \log \theta : \theta \geq 1 \}$. 
Then the $g$-dual of $v$ is given by $v^g(x)=\log ({1+x})/2$, $x 
\geq 0$ and is a candidate for $\psi$. However $e^{v^g(x)} = 
\sqrt{1+x}$ is not convex. 
There is no diffusion in $\sX_0$ consistent with $V$. 
\end{example}

\begin{example}
\label{ex:martin}
Forward and Inverse Problem:
In special cases, the optimal strategy in the forward problem may be to `stop at the 
first hitting time of infinity' or to `wait forever'. Nonetheless, it is possible to solve the forward and 
inverse problems.

Let $h$ be an increasing, differentiable 
function on $[0,\infty)$ with $h(0)=1$,
such that $e^h$ is convex;
let $f$ be a positive, increasing, differentiable function on $[0,\infty)$ 
such that $\lim_{x \rightarrow \infty} f(x) = 1$; and let $w(\theta)$ 
be 
a non-negative, increasing and differentiable function on $\Theta=[\theta_-,\theta_+]$ with $w(\theta_-)=0$.

Suppose that 
\[g(x,\theta)=h(x)+f(x)w(\theta).\]
Note that the cross-derivative $g_{x \theta}(x,\theta)=f'(x)w'(\theta)$ is non-negative.

Consider the forward problem. Suppose we are given a diffusion in 
$\sX_0$ 
with log-eigenfunction $\psi=h$. Then the log-problem value $v$ is 
given by 
\[ v(\theta) = \psi^g(\theta)=\sup_{x \geq 0} \{g(x,\theta)-\psi(x)\} = 
\limsup_{x 
\rightarrow \infty} \{f(x) w(\theta)\}= w(\theta).\]

Conversely, suppose we are given the value function $V=e^w$ on
$\Theta$. Then
\[ w^g(x)=\sup_{\theta \in \Theta} \{g(x,\theta)-w(\theta)\} = 
\sup_{\theta \in \Theta} \{h(x) + (f(x)-1)w(\theta)\} =
h(x) \]
is the log-eigenfunction of a diffusion $X \in \sX_0$ which solves the 
inverse problem.
\end{example}

A generalised diffusion
$X \in \sX_0$ can be identified by its speed measure $m$. Let $m$ be a
non-negative,
non-decreasing and right-continuous function which defines a measure on
$\R^+$, and let $m$ be identically zero on $\R^-$. We call $x$ a point
of growth of $m$ if $m(x_1)<m(x_2)$ whenever
$x_1 < x < x_2$ and denote the closed set of points of growth by $E$.
Then $m$ may assign mass
to 0 or not, but in either case we assume $0 \in E$.
We also assume that if $\xi=\sup \{x : x \in E \}$
then $\xi+m(\xi+) = \infty$. If $\xi < \infty$ then either $\xi$ is an 
absorbing endpoint, or $X$ does not reach 
$\xi$ in finite time.

The diffusion $X$ with speed measure $m$ is defined on $[0,\xi)$ and is 
constructed via a time-change of Brownian motion as follows. 

Let $\F^B=({\mathcal F}_u^B)_{u\geq 0}$ be a filtration supporting
a Brownian Motion $B$ started  at $0$ with a local time process
$\{ L_u^z ; u \geq 0, z \in \R \}$. Define $\Gamma$ to be the
left-continuous,
increasing, additive functional
\[\Gamma_u = \int_{\R} L_u^z m(dz),\]
and define its right-continuous inverse by
\[A_t = \inf \{u : \Gamma_u > t \}. \]
If we set $X_t = B(A_t)$ then $X_t$ is a generalised diffusion which is
a local martingale away from 0, and which is absorbed the first time
that $B$ hits $\xi$.

For a given diffusion $X \in \sX_0$ recall that $\varphi(x) \equiv 
\varphi_X(x)$
is defined via $\varphi_X(x) = ( \E_0[e^{-\rho H_x}])^{-1}$. It is well 
known (see for example
\cite[V.50]{rogers} and \cite[pp 147-152]{dym}) that $\varphi_X$ is the 
unique
increasing, convex solution to the differential equation
\begin{equation} \label{eq:differential}
\frac{1}{2} \frac{d^2 f}{dm dx} = \rho f; \hspace{20mm} f(0)=1, 
\hspace{10mm} f'(0-)=0.
\end{equation}
Conversely, given an increasing convex function $\varphi$ with
$\varphi(0)=1$ and $\varphi'(0+) \geq 0$, (\ref{eq:differential}) can be 
used to define a measure $m$ which in turn is the speed measure of a 
generalised diffusion $X \in \sX_0$.

If $m(\{x\}) > 0$ then the process $X$ spends a positive amount of time 
at $x$. If $x \in E$ is an isolated point, then there is a positive 
holding time at $x$, conversely, if for each neighbourhood $N_x$ of $x$, 
$m$ also assigns positive mass to $N_x \setminus \{ x \}$, then $x$ is a 
sticky point.

If $X \in \sX_0$ and $m$ has a density, then 
$m(dx)=\sigma(x)^{-2}dx$ where $\sigma$ is the diffusion 
coefficient of $X$ and the differential equation 
(\ref{eq:differential}) becomes
\begin{equation}\label{eq:differentialnice}
\frac{1}{2} \sigma(x)^2 f''(x) - \rho f(x) = 0.
\end{equation}
In this case, depending on the smoothness of $g$, 
$v$ will also inherit smoothness properties. Conversely, `nice' $v$ will 
be associated with processes solving (\ref{eq:differentialnice}) for a smooth 
$\sigma$. However, rather than pursuing issues of regularity, we prefer 
to work with generalised diffusions. 

\section{u-convex Analysis}
\label{S:uconvex}

In the following we will consider $u$-convex functions for $u=u(y,z)$ a  
function of 
two variables $y$ and $z$. There will be complete symmetry in role 
between $y$ and $z$ so that although we will discuss $u$-convexity for 
functions of $y$, the same ideas apply immediately to $u$-convexity 
in the variable $z$. Then, in the sequel we will apply these results for 
the function $g$, and we will apply them for $g$-convex functions of 
both $x$ and $\theta$.

For a more detailed development of $u$-convexity, see 
\cite{rachev}, \cite{villani}, \cite{carlier} and the references 
therein. Proofs of the results below are included in the Appendix.

Let $D_y$ and $D_z$ be sub-intervals of $\R$. We suppose that $u : D_y \times
D_z \mapsto \bar{\R}$ is well defined, though possibly infinite valued.

\begin{definition} \label{def:uconvex1} $f: D_y \rightarrow \R^+$ is 
$u$-convex iff there exists a non-empty $S \subset D_z \times \R$ 
such that for all $y \in D_y$
\[f(y)=\sup_{(z,a)\in S} [u(y,z)+a] . \]
\end{definition}

\begin{definition}
The $u$-dual of $f$ is the $u$-convex function on $D_z$ given by 
\[ f^u(z)=\sup_{y \in D_y} [u(y,z)-f(y)] . \]
\end{definition}

A fundamental fact from the theory of $u$-convexity is the following:
\begin{lemma} \label{def:uconvex2}
A function $f$ is $u$-convex iff $(f^u)^u=f$. 
\end{lemma}

The function $(f^u)^u$ (the $u$-convexification of $f$) is the greatest 
$u$-convex 
minorant of $f$ (see the Appendix). 
The condition $(f^u)^u=f$ provides an alternative
definition of a $u$-convex function, and is often 
preferred; checking whether $(f^u)^u=f$ is 
usually more natural than trying to identify the set $S$.

Diagrammatically (see Figure 1.), we can think of 
$-(f^u)(z)=\inf_y[f(y)-u(y,z)]$ 
as the vertical distance between $f$ and $u(.,z)$. Thus 
$f^u(z) \leq 0$ when $f(y) \geq u(y,z)$ for all $y \in D_y$. 

The following description due to Villani \cite{villani} is helpful in 
visualising what is going on: $f$ is $u$-convex if at every point $y$ we 
can find a parameter $z$ so that we can {\it caress} $f$ from below with 
$u(.,z)$.

The definition of the $u$-dual implies a generalised version of the 
Young inequality (familiar from convex analysis, e.g 
\cite{rockafellar}),
\[f(y)+f^u(z) \geq u(y,z)\] 
for all $(y,z) \in D_y \times D_z$. Equality holds at pairs $(y,z)$ where the supremum 
\[\sup_z [u(y,z)-f^u(z)]\] 
is achieved.

\begin{definition}\label{def:subdifferential}
The $u$-subdifferential of $f$ at $y$ is defined by
\[\partial^u f(y) = \{ z \in D_z : f(y) + f^u(z) = u(y,z) \}, \]
or equivalently
\[\partial^u f(y) = \{ z \in D_z : u(y,z)- f(y) \geq u(\hat{y},z) - 
f(\hat{y}), \forall \hat{y} \in D_y \} . \]
\end{definition}

If $U$ is a subset of $D_y$ then we define $\partial^u f(U)$ to be 
the union of $u$-subdifferentials of $f$ over all points in $U$.

\begin{definition}
$f$ is $u$-subdifferentiable at $y$ if $\partial^u f(y) \neq 
\emptyset$. $f$ is $u$-subdifferentiable on $U$ if it is 
$u$-subdifferentiable 
for all $y \in U$, and $f$ is $u$-subdifferentiable if it is
$u$-subdifferentiable
on $U=D_y$.
\end{definition}

\begin{figure}[t]\label{Fig.1}
\begin{center}
\includegraphics[height=8cm,width=9cm]{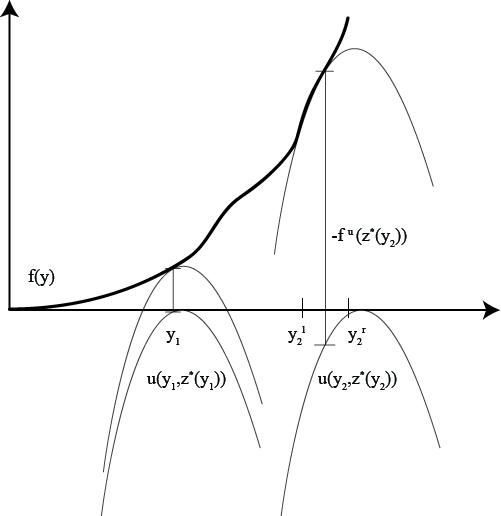}
\caption{$f$ is $u$-subdifferentiable. $\partial^u 
f(y_1)=\zstar(y_1)$ and $\partial^u (y_2)=\zstar(y_2)$ for $y_2 \in 
(y_2^l,y_2^r)$. The distance between $u(.,z)$ and $f$ is equal to 
-$f^u(z)$. Note that the $u$-subdifferential is constant over the 
interval $(y_2^l,y_2^r)$}
\end{center}
\end{figure}

In what follows it will be assumed that the function $u(y,z)$ is
satisfies the following `regularity
conditions'.

\begin{assumption} \label{ass}
\begin{enumerate}
%\item[A1.] For each $z$, $u(y,z)$ is increasing as a function
%of $y$, conversely, for each $y$, $u(y,z)$ is increasing as .
\item[(a)] $u(y,z)$ is continuously twice differentiable.
\item[(b)] $u_y(y,z) =\frac{\partial}{\partial y} u(y,z)$ as a function 
of $z$, 
and $u_z (y,z) =\frac{\partial}{\partial z} u(y,z)$ as a function of 
$y$, are strictly increasing.
\end{enumerate}
\end{assumption}

\begin{remark}
We will see below that by assuming \ref{ass}(a) 
irregularities
in the value function (\ref{eq:forward}) can be identified with extremal
behaviour of the diffusion. 
\end{remark}

\begin{remark}
Condition \ref{ass}(b) is known as the single crossing property and as 
the
Spence-Mirrlees condition (\cite{carlier}). If instead we have
the `Reverse Spence-Mirrlees condition': \\
\hspace{10mm} (bb)  $u_y(y,z)$ as a function of 
$z$,
and $u_z (y,z)$ as a function of $y$, are strictly decreasing, \\ then 
there 
is a parallel theory, see Remark~\ref{RSM}. 
\end{remark}

The following results from $u$-convex analysis will be fundamental in 
our application of $u$-convex analysis to finding the solutions of the 
forward and inverse problems.

\begin{lemma} \label{lem:monsubdiff}
Suppose $f$ is $u$-subdifferentiable, and $u$ satisfies 
Assumption~\ref{ass}.
Then $\partial^u f$ is monotone in the following 
sense: \\
Let  $y,\hat{y} \in D_y$, $\hat{y} > y$.
Suppose $\hat{z} \in \partial^u
f(\hat{y})$ and $z \in
\partial^u f(y)$.
Then $\hat{z} \geq z$.
\end{lemma}

\begin{definition} \label{d:310}
We say that a function is strictly $u$-convex, when its
$u$-subdifferential is strictly monotone.
\end{definition}

\begin{proposition} \label{p:prop1} 
Suppose that $u$ satisfies Assumption~\ref{ass}.

Suppose $f$ is a.e differentiable and 
$u$-subdifferentiable. Then there exists a map $\zstar: D_y 
\rightarrow D_z$ such that if $f$ is differentiable at $y$ then
$f(y)=u(y,\zstar(y))-f^u(\zstar(y))$ and
\begin{equation} \label{eq:diffsubdiff}
f '(y) = u_y(y,\zstar(y)).
\end{equation}

Moreover, 
$\zstar$ is such 
that  
$\zstar(y)$ is non-decreasing.

Conversely, suppose that $f$ is a.e 
differentiable and equal to the integral of its derivative. If 
(\ref{eq:diffsubdiff}) holds for a non-decreasing function
$\zstar(y)$, then $f$ is $u$-convex and 
$u$-subdifferentiable with $f(y)=u(y,\zstar(y))-f^u(\zstar(y))$.
\end{proposition}

Note that the subdifferential $\partial^u f(y)$ may be an interval in 
which case $z^*(y)$ may be taken to be any element in that interval.
Under Assumption~\ref{ass}, $\zstar(y)$ is non-decreasing

We observe that since $u(y,\zstar(y))=f(y)+f^u(\zstar(y))$ we have 
$u(y^*(z),z)=f(y^*(z))+f(z)$ and $y^*(z) \in \partial^u f^u(z)$ so 
that $y^*$ may be defined directly as an element of $\partial^u f^u$.
If $\zstar$ is strictly increasing then $y^*$ is just the inverse of $\zstar$.

\begin{remark}\label{RSM}
If $u$ satisfies the `Reverse Spence-Mirrlees' condition, the conclusion 
of 
Lemma~\ref{lem:monsubdiff} is unchanged except that now `$z \geq 
\hat{z}$'.
Similarly, Proposition~\ref{p:prop1} remains true, except that 
$z^*(y)$ and $y^*(z)$ are non-increasing.
\end{remark}

\begin{proposition} \label{p:prop2}
Suppose that $u$ satisfies Assumption~\ref{ass}.

Suppose $f$ is 
$u$-subdifferentiable in a neighbourhood of $y$. Then 
$f$ is continuously differentiable at $y$ 
if and only if $z^*$ is continuous at $y$.
\end{proposition}

\section{Application of $u$-convex analysis to the Forward Problems}

Now we return to the context of the family of optimal control problems
(\ref{eq:forward}) and the representation (\ref{eq:logvupsi}).

\begin{lemma} \label{l:lipschitzphi} 
Let $X \in \sX_0$ be a diffusion in natural scale
reflected at the origin with a finite or infinite right boundary point
$\xi$. Then the increasing $log$-eigenfunction of the generator 
\[\psi_X(x) = -\log(\E[e^{- \rho H_x}]^{-1})\] 
is locally Lipschitz continuous on $(0,\xi)$. 
\end{lemma}

\begin{proof} $\varphi_X(x)$ is increasing, convex and finite and
therefore locally Lipschitz on $(0,\xi)$. $\varphi(0)=1$, and since 
$\log$
is locally Lipschitz on $[1,\infty)$, $\psi=\log(\varphi)$ is locally
Lipschitz on $(0,\xi)$. \end{proof}

Henceforth we assume that $g$ satisfies Assumption~\ref{ass}, so that 
$g$ is twice differentiable and satisfies the Spence-Mirrlees condition. 
We assume further
that $G(x,\theta)$ is non-decreasing in $x$. Note that this is without 
loss of generality since it can never be 
optimal
to stop at $x'>x$ if $G(x',\theta) < G(x,\theta)$, since to wait
until the first hitting time of $x'$ involves greater discounting and a lower 
payoff.

Consider the forward problem. Suppose the aim is to solve 
(\ref{eq:logvupsi}) for a given $X \in \sX_0$ with associated 
$\log$-eigenfunction $\psi(x) = \psi_X(x) = - \log \E_0[e^{- \rho H_x}]$ 
for the family of objective functions $\{ G(x,\theta) : \theta \in 
\Theta \}$. Here $\Theta$ is assumed to be an interval with endpoints 
$\theta_-$ and $\theta_+$, such that $\Theta \subseteq D_\theta$.

Now let 
\begin{equation}\label{eq:forwardsubdiff}
v(\theta)=\sup_{x : \psi(x) < \infty} [g(x,\theta)-\psi(x).]
\end{equation}
Then
$v=\psi^g$ is the $g$-convex dual of $\psi$. 

By definition $\partial^g v(\theta)=\{x:v(\theta)=g(x,\theta)-\psi(x)\}$ is the (set of) level(s)
at which it is optimal to stop for the problem parameterised by $\theta$. If $\partial^g v(\theta)$
is empty then there is no optimal stopping strategy in the sense that for any finite stopping rule there
is another which involves waiting longer and gives a higher problem value. 

Let $\theta_R$ be the infimum of those values of $\theta \in \Theta$ 
such that 

$\partial^gv(\theta)=\emptyset$. If $v$ is nowhere
$g$-subdifferentiable then we set $\theta_R= \theta_-$. 

\begin{lemma}\label{l:subdiffforward}
The set where 
$v$ is $g$-subdifferentiable 
forms an interval with endpoints $\theta_-$ and $\theta_R$.
\end{lemma}

\begin{proof}
Suppose $v$ is $g$-subdifferentiable at $\hth$, and suppose $\theta \in 
(\theta_-,\hth)$. We claim that $v$ is $g$-subdifferentiable at 
$\theta$.

Fix $\hx \in \partial^g v(\hth)$. Then $v(\hth) = 
g(\hx,\hth) - \psi(\hx)$ and

\begin{equation}
\label{eq:subd1}
g(\hx,\hth) - \psi(\hx) \geq  g(x,\hth) - \psi(x),
\hspace{20mm} \forall x < \xi,
\end{equation}
and for $x = \xi$ if $\xi<\infty$. We write the remainder of the proof
as if we are in the case $\xi<\infty$; the case $\xi=\infty$ involves 
replacing $x \leq \xi$ with $x < \xi$.

Fix $\theta<\hth$. We want to show
\begin{equation}
\label{eq:subd2}
g(\hx,\theta) - \psi(\hx) \geq  g(x,\theta) - \psi(x),
\hspace{20mm} \forall x \in (\hx, \xi],
\end{equation}
for then
\[ \sup_{x \leq \xi} \{ g(x,\theta) - \psi(x) \} =  \sup_{x \leq \hx}  
\{ g(x,\theta) - \psi(x) \},
\]
and since $g(x,\theta) - \psi(x)$ is continuous in $x$ the supremum is 
attained.

By assumption, $g_\theta(x,t)$ is increasing in $x$, and so for $x \in 
(\hx,\xi]$
\[ \int_\theta^{\hth} [g_\theta(\hx,t) - g_\theta(x,t)] dt \leq 0 \]
or equivalently,
\begin{equation}
\label{eq:subd3}
g(\hx,\hth) - g(\hx,\theta) \leq  g(x,\hth) - g(x,\theta).
%\hspace{20mm} \forall x \in (\hx, \xi),
\end{equation}
Subtracting (\ref{eq:subd3}) from (\ref{eq:subd1}) gives 
(\ref{eq:subd2}).
\end{proof}

\begin{lemma} \label{l:vlipschitz}
$v$ is locally Lipschitz on $(\theta_-,\theta_R)$. 
\end{lemma}

\begin{proof}
On $(\theta_-,\theta_R)$ $v(\theta)$ is $g$-convex, 
$g$-subdifferentiable and $\xstar(\theta)$ is monotone increasing.

Fix $\theta', \theta''$ such that $\theta_- < \theta'' < \theta' < 
\theta_R$.
Choose $x' \in \partial^g v(\theta')$ and $x'' \in \partial^g 
v(\theta'')$
and suppose $g$ has Lipschitz constant $K'$ (with respect to $\theta$) 
in 
a neighbourhood of $(x',\theta')$.

Then $v(\theta') = g(x',\theta') - \psi(x')$ and
$v(\theta'') \geq g(x',\theta'') - \psi(x')$
so that 
\[ v(\theta') - v(\theta'') \leq g(x', \theta')-g(x',\theta'') 
\leq K'(\theta' - \theta'') \]
and a reverse inequality follows from considering $v(\theta'') = 
g(x'',\theta'') - \psi(x'')$.
\end{proof}

Note that it is not possible under our assumptions to date ($g$ 
satisfying Assumption~\ref{ass}, and $g$ monotonic in $x$) to conclude 
that $v$ is continuous at $\theta_-$, or even that $v(\theta_-)$ exists.
Monotonicity guarantees that even if $\theta_- \notin \Theta$ we can 
still define
$\xstar(\theta_-) := \lim_{\theta \downarrow \theta_-} \xstar(\theta)$.
%\begin{example}
For example, suppose $\Theta=(0,\infty)$ and for $\epsilon \in (0,1)$ 
let $g_{\epsilon}(x,\theta) = g(x,\theta) + \epsilon f(\theta)$. Then if
$v_{\epsilon}(\theta)$ is the $g_{\epsilon}$-convex dual of $\psi$ we 
have $v_{\epsilon}(\theta) = v(\theta) + \epsilon f(\theta)$, where
$v(\theta) = v_{0}(\theta)$. If $g$ and $\psi$ are such that 
$\lim_{\theta \downarrow 0} v(\theta)$ exists and is finite, then 
choosing any 
bounded $f$ for which $\lim_{\theta \downarrow 0} f(\theta)$ does 
not exist gives an example for which $\lim_{\theta \downarrow 0} 
v_{\epsilon}(\theta)$ does not exist.   
%\end{example}
It is even easier to construct modified examples such that $v(\theta_-)$
is infinite.

Denote $\Sigma(\theta,\xi)=\limsup_{x \uparrow \xi} 
\{ g(x,\theta)-\psi(x) \}$. 
Then for $\theta_R<\theta <\theta_+$, 
$\psi^g(\theta)=\Sigma(\theta,\xi)$. We 
have shown:
\begin{proposition}\label{p:forward}
If $g$ satisfies Assumption~\ref{ass}, $g$ is increasing in $x$ and if $X$ is a reflecting diffusion 
in natural scale then the solution to the forward problem is
$ %\[
V(\theta)= 
%\left\{\begin{array}{ll}
\exp(\psi^g(\theta)) %&\; \theta \in (\theta_-, \theta_R) \\
%\exp(\Sigma(\theta,\xi)) &\; \theta_R \leq \theta < \theta^+ .
%\end{array}\right.
$. %\]
%{\bf What happens if $\theta_- = \theta_R$} 
\end{proposition}

\begin{remark}\label{rem:decr}

Suppose now that $g_x(x,.)$ is strictly decreasing (the reverse 
Spence-Mirrlees condition). The arguments above apply with the obvious 
modifications. Let $\theta_L$ be the supremum of those values $\theta 
\in \Theta$ 
such 
that $\xstar(\theta)=\emptyset$. Then the analogues to Lemmas 
\ref{l:subdiffforward} and \ref{l:vlipschitz} show that $v$ is 
$g$-subdifferentiable and locally Lipschitz on $(\theta_L,\theta_+)$ 
and 
that for $\theta_- <\theta<\theta_L$
\[
V(\theta)= %\left\{\begin{array}{ll}
\exp(\Sigma(\theta,\xi)) . %&\; \theta_- \leq \theta \leq \theta_L \\
%\exp(\psi^g(\theta)) &\; \theta \in (\theta_L, \theta_+) .
%\end{array}\right.
\]
\end{remark}

We close this section with some examples. 
\begin{example}\label{ex:notlipsch}
Recall Example~\ref{ex:powers}, but note that in that example 
$\theta$ was restricted to take values in $\Theta=(1,2)$.
Suppose $\Theta = [0,\infty)$, $g(x,\theta) = \theta \log x$ and 
$\psi(x)=\log 
(1+x^2)$. Then $\theta_R=2$ and for $\theta< \theta_R$, 
$x^*(\theta)=(\theta/(2-\theta))^{1/2}$. Further, for $\theta \leq 2$ 
\[v(\theta) = 
\frac{\theta}{2} 
\log(\theta) + \frac{2 - \theta}{2} \log (2-\theta) - \log 2, \] 
and $v(\theta)=\infty$ for $\theta > 2$.

Note that $v$ is continuous on $[0,\theta_R]$, but not on $\Theta$.

\end{example}

\begin{example}\label{ex:fullcalc} Suppose $g(x,\theta) = x 
\theta$ and $\Theta = (0,\infty)$. Suppose $X$ is a diffusion on 
$[0,1)$, 
with $1$ a 
natural boundary and diffusion coefficient 
$\sigma(x)^2=\frac{\rho (1-x^2)^2}{1+x^2}$. Then 
$\varphi(x)=\frac{1}{1-x^2}$ 
and 
\[v(\theta)=\sup_{x<1} [\theta x + \log(1-x^2) ]. \] 
It is 
straightforward to calculate 
that $\xstar(\theta)= \sqrt{1 + \theta^{-2}} - 1/\theta$ and then 
that $v(\theta) : (0,\infty)  \rightarrow \R$ is given by 
\begin{equation}
\label{eq:vfullcalc}
v(\theta)=\sqrt{1+\theta^2}-1-\log \left( 
\frac{\theta^2}{2(\sqrt{1+\theta^2}-1)} \right) . 
\end{equation}
\end{example}

\section{Application of $u$-convex analysis to the Inverse Problem}
\label{S:inverse}

Given an interval $\Theta \subseteq \Rex$ with endpoints
$\theta_-$ and $\theta_+$ and a value function $V$ defined on $\Theta$
we now discuss how to determine whether or not 
there exists a diffusion in $\sX_0$ that solves the inverse problem for 
$V$. Theorem~\ref{t:existence} gives a necessary and sufficient 
condition for existence. This condition is rather indirect, so in
Theorem~\ref{t:suffexist} we give some sufficient conditions in terms 
of the $g$-convex dual $v^g$ and associated objects.

Then, given existence, a supplementary question is whether 
$\{V(\theta): \theta \in 
\Theta\}$ contains enough information to determine the diffusion 
uniquely. In Sections~\ref{SS:leftextension}, \ref{SS:rightextension}
and \ref{SS:middleextension} we consider three different phenomena 
which 
lead to non-uniqueness. Finally in Section~\ref{SS:uniqueness} we give 
a simple sufficient condition for uniqueness.

Two key quantities in this section are the lower and upper bound for the 
range of the $g$-subdifferential of $v$ on $\Theta$. Recall that we 
are assuming that the Spence-Mirrlees condition holds so that $\xstar$ 
is 
increasing on $\Theta$.
Then, if $v$ is somewhere $g$-subdifferentiable we set $\xminus 
= \sup \{ x \in \partial^g v(\theta_-) \}$, or if $\theta_- \notin 
\Theta$,  
$\xminus=\lim_{\theta \downarrow \theta_-} \xstar(\theta)$. Similarly, 
we 
define $\xplus= \inf \{ x \in \partial^g v(\theta_+) \}$, 
or if $\theta_+ \notin \Theta$, 
$\xplus=\lim_{\theta \uparrow \theta_+} \xstar(\theta)$, and
$\xR=\lim_{\theta \uparrow \theta_R} \xstar(\theta)$.
If $v$ is nowhere $g$-subdifferentiable then we set 
$\xminus=\xR=\xplus=\infty$.

%Previously we established that when the Spence-Mirrlees condition holds 
%then $\xstar$ is increasing on $\Theta$. It follows that information 
%about $V$ on $\Theta$ translates into information about consistent 
%diffusions on $[\xminus, \xi)$. The conditions for existence and 
%uniqueness of diffusions solving the inverse stopping problem depend on 
%the form of the interval $[\xminus, \xi)$.

\subsection{Existence}
\label{SS:existence}
In the following we assume that $v$ is $g$-convex on $\R_+ \times \Theta$, 
which means that for all $\theta \in \Theta$,
\[v(\theta)={v^g}^g(\theta)=\sup_{x \geq 0} \{g(x,\theta)-v^g(x)\}.\]
Trivially this is a necessary condition for the existence of a diffusion
such that the solution of the optimal stopping problems are given by 
$V$. Recall that we are also assuming that $g$ is increasing in $x$ and 
that it satisfies Assumption~\ref{ass}.

The following fundamental theorem provides necessary and sufficient 
conditions for existence of a consistent diffusion.

\begin{theorem} \label{t:existence}
There exists $X \in \sX_0$ such that $V_X=V$ if and only if there exists 
$\phi : [0,\infty) \rightarrow [1,\infty]$ such that $\phi(0)=1$, $\phi$
is increasing and convex and $\phi$ is such that $(\log \phi)^g=v$ on 
$\Theta$. 
\end{theorem}

\smallskip
\begin{proof} 

If $X \in \sX_0$ then $\phi_X(0)=1$ and $\phi_X$ is increasing and convex. 
Set $\psi_X=\log{\phi_X}$. If $V_X=V$ then \[v(\theta)=v_X(\theta)=\sup_x \{g(x,\theta)-\psi_X(x)\}= \psi^g_X.\]
  
Conversely, suppose $\phi$ satisfies the conditions of the theorem, and 
set $\psi = \log \phi$. Let 
$\xi = \sup \{ x : \phi(x) < \infty \}$.  
Note that if $\xi < \infty$ then
\[ (\log \phi)^g(\theta) = \sup_{x \geq 0} \{g(x,\theta)- \psi(x) \} = 
\sup_{x \leq \xi} \{g(x,\theta)-\psi(x) \} \]
and the maximiser $x^*(\theta)$ satisfies $x^*(\theta) \leq \xi$.

For $0\leq x \leq \xi$ define a measure $m$ via
\begin{equation} \label{eq:speed}
m(dx)= \frac{1}{2 \rho} \frac{\phi''(x)}{\phi}  dx  
     = \frac{\psi''(x)+ (\psi'(x))^2}{2 \rho} dx.
\end{equation}
Let $m(dx)=\infty$ for $x<0$, and, if $\xi$ is finite 
$m(dx) = \infty$ for $x > \xi$. We 
interpret (\ref{eq:speed}) in a distributional sense whenever $\phi$ 
has a discontinuous derivative. In the language of strings $\xi$ is the 
length of the string with mass distribution $m$. We assume that
$\xi>0$. The case $\xi = 0$ is a degenerate case which can be covered 
by a direct argument. 

Let $B$ be a Brownian motion started at 0 with local time process $L_u^z$ 
and define $(\Gamma_u)_{u \geq 0}$ via
\[\Gamma_u=\int_{\R} m(dz) L_u^z = \int_0^t \frac{1}{2 \rho} 
\frac{\psi''(B_s)+ (\psi'(B_s))^2}{\psi(B_s)} ds .\]
Let $A$ be the right-continuous inverse to $\Gamma$. Now
set $X_t=B_{A_t}$. Then $X$ is a local martingale (whilst away from zero) 
such that $d \langle  X \rangle_t/dt = dA_t/dt = (dm/dx |_{x = X_t})^{-1}$.
When $m(dx)= \sigma(x)^{-2} dx$, we have $d\langle  X \rangle_t = 
\sigma(X_t)^2 dt$.

We want to conclude that $\E[e^{-\rho H_x}]=\exp(-\psi(x))$. 
Now, $\varphi_X(x)=(\E[e^{-\rho H_x}])^{-1}$ is the 
unique increasing solution to 
\[\frac{1}{2} \frac{d^2 f}{dm dx} = \rho f\,\] 
with the boundary conditions $f'(0-)=0$ and $f(0)=1$. 
Equivalently, for all $x,y \in (0,\xi)$ with $x<y$, 
$\varphi_X$ solves \[f'(y-)-f'(x-)= \int_{[x,y)} 2 \rho f(z) m(dz).\] By 
the definition of $m$ above it is easily verified that $\exp(\psi(x))$ 
is a solution to this equation. 
Hence $\phi = \varphi_X$ and our candidate process solves the inverse 
problem.
\end{proof}

\begin{remark}
\label{r:candidate}
Since $v$ is $g$-convex a natural candidate for $\phi$ is $e^{v^g(x)}$, 
at least if $v^g(0)=0$ and $e^{v^g}$ is convex. Then $\phi$ is the 
eigenfunction $\varphi_X$ of a diffusion $X \in \sX_0$.
\end{remark}

Our next example is one where $\phi(x) = e^{v^g(x)}$ is convex but not 
twice differentiable, and in consequence the consistent diffusion has a 
sticky point. This illustrates the need to work with generalised 
diffusions. For related examples in a different context see Ekstr\"{o}m 
and Hobson~\cite{hobson}.

\begin{example}\label{ex:stickyatpoints}
Let $\Theta = \R_+$ and let the objective function be
$g(x,\theta)=\exp(\theta x)$.
Suppose
\[
V(\theta)= \left\{\begin{array}{ll}
\exp(\frac{1}{4} \theta^2) &\; 0 \leq \theta \leq 2, \\
\exp(\theta - 1) &\; 2 < \theta \leq 3, \\
\exp(\frac{2}{3 \sqrt{3}} \theta^{3/2}) &\; 3 < \theta.
\end{array}\right.
\]
Writing $\varphi = e^{v^g}$ we calculate
\[\varphi(x)= \left\{\begin{array}{ll}
\exp(x^2) &\; 0 \leq x \leq 1, \\
\exp(x^3) &\; 1 < x.
\end{array}\right.
\]
Note that $\varphi$ is increasing and convex, and $\varphi(0)=1$.
Then $\varphi '$ jumps at $1$ and since
\[ \varphi(1) = \varphi ' (1+)- \varphi '(1-) = 2 \rho \varphi(1) m(\{1 
\})\]
we conclude that $m (\{ 1 \}) =\frac{1}{2 \rho}$. Then $\Gamma_u$
includes a multiple of the local time at
1 and the diffusion $X$ is sticky there.
\end{example}

Theorem~\ref{t:existence} converts a question about existence of a 
consistent diffusion into a question about existence of a 
log-eigenfunction with particular properties including $(\log \phi)^g = 
v$. We would like to have conditions which apply more directly to the 
value function $V(\cdot)$.
The conditions we derive
depend on the value of $\xminus$. 
%Recall that we are assuming that $v$ is $g$-convex on $\R_+ \times 
%\Theta$. Since a further necessary condition is that $v$ is locally 
%Lipschitz (Lemma~\ref{l:vlipschitz}), then we assume that also.

As stated in Remark~\ref{r:candidate}, a natural candidate for $\phi$ 
is $e^{v^g(x)}$. As we prove below, if $\xminus=0$ this candidate leads 
to a consistent diffusion provided $v^g(0)=0$ and $e^{v^g(x)}$ is 
convex and strictly increasing. If $\xminus>0$ then the 
sufficient conditions are 
slightly different, and $e^{v^g}$ need not be globally convex.

\begin{theorem} \label{t:suffexist}
Assume $v$ is $g$-convex.
Each of the following is a sufficient condition for there to exist a 
consistent diffusion:
\begin{enumerate}
\item $\xminus=0$, $v^g(0)=0$ and $e^{v^g(x)}$ is convex and 
increasing on $[0,\xplus)$.
\item $0<\xminus<\infty$, $v^g(\xminus)>0$, $e^{v^g(x)}$ is convex and 
increasing on 
$[\xminus,\xplus)$, %$L'(\xminus) \leq (e^{v^g(\xminus+)})'$ 
and on $[0,\xminus)$, $v^g(x) \leq f(x)=\log(F(x))$ where 
\[F(x)=1+x \frac{\exp(v^g(\xminus))-1}{\xminus}\] 
is the straight line connecting the points $(0,1)$ 
and $(\xminus,e^{v^g(\xminus)})$.
\item $\xminus=\infty$ and there exists a convex, 
increasing function $F$ with $\log(F(0))=0$ 
such that $f(x) \geq v^g(x)$ for all $x \geq 0$ 
and \[\lim_{x \rightarrow \infty}\{f(x)-v^g(x) = 0\},\]
where $f = \log F$.
\end{enumerate}
\end{theorem}

\begin{proof}

We treat each of the conditions in turn. If $\xminus=0$ 
then Theorem \ref{t:existence} applies directly on taking $\phi(x) = 
e^{v^g(x)}$, with $\phi(x)=\infty$ for $x>\xplus$ (we use the fact that
$v$ is $g$-convex and so ${v^g}^g=v$).

Suppose $0<\xminus<\infty$. The condition $e^{v^g(x)} \leq F(x)$ on 
$[0,\xminus)$ implies $F'(x)= (e^{v^g(\xminus)} - 1)/\xminus \leq 
(e^{v^g(\xminus -)})'$. Although the left-derivative ${v^g(x-)}'$ need not 
equal the right-derivative ${v^g(x+)}'$ the arguments in the proof of 
Proposition~\ref{p:prop1} show that ${v^g(x-)}' \leq 
v^g(x+)'$.
This implies that the function
\[\phi_F(x) = \left\{ \begin{array}{ll} 
F(x) & x < \xminus \\
\exp(v^g(x))        & \xminus \leq x <\xplus
\end{array} \right. \]
is convex at $\xminus$ and hence convex and increasing on $[0,\xplus)$.

Setting $\phi_F(\xplus) = \lim_{x \uparrow \xplus} \phi_F(x)$ and 
$\phi_F = 
\infty$ for $x>\xplus$ 
we have a candidate for the 
function in Theorem~\ref{t:existence}.

It remains to show that $(\log \phi_F)^g = v$ on $\Theta$.
We now check that $\phi_F$ is consistent with $V$ on $\Theta$, which 
follows if 
the $g$-convex dual of $\psi=\log(\phi_F)$ is equal to $v$ on $\Theta$.

Since $\psi \geq v^g$ we have $\psi^g \leq v$. We aim to prove the 
reverse inequality.
By definition, we have for $\theta \in \Theta$
\begin{equation} \label{eq:vee}
\psi^g(\theta)=\left( \sup_{x < \xminus}\{g(x,\theta)-f(x)\} \right) 
\vee 
\left( \sup_{\xminus \leq x \leq \xplus}\{g(x,\theta)-v^g(x)\}\right).
\end{equation}

Now fix $x \in [0,\xminus)$. For $\theta<\theta_R$ we have 
by the definition of the $g$-subdifferential
\[g(\xstar(\theta),\theta)-v^g(\xstar(\theta)) \geq g(x,\theta)-v^g(x).\] 
Hence $v(\theta) = \sup_{x \geq 0}\{g(x,\theta)-v^g(x)\} = \sup_{x 
\geq \xminus}\{g(x,\theta)-v^g(x)\} \leq \psi^g(\theta)$.

Similarly, if $\theta \geq \theta_R$ we have for all $x' \in [0,\xminus)$,
\[\limsup_{x \rightarrow \infty} g(x,\theta)-v^g(x) 
\geq g(x',\theta)-v^g(x').\] 
and $v(\theta) = \lim \sup_{x}\{g(x,\theta)-v^g(x)\} = \sup_{x 
\geq
\xminus }\{g(x,\theta)-v^g(x)\} \leq \psi^g(\theta)$.

Finally, suppose $\xminus = \infty$. 
By the definition of $f^g$ and the 
condition 
$f \geq v^g$ we get
\begin{eqnarray*}
f^g(\theta) &=& \sup_{x \geq 0}\{g(x,\theta)-f(x)\} \\
&\leq& \sup_{x \geq 0} \{g(x,\theta)-v^g(x)\} \\
&=& v(\theta).
\end{eqnarray*}

On the other hand
\begin{eqnarray*}
v(\theta) 
& = &\limsup_{x \rightarrow \infty}\{g(x,\theta)- f(x) + f(x) - v^g(x) 
\} \\
& \leq  &\limsup_{x \rightarrow \infty}\{g(x,\theta)-f(x)\} +
\lim_{x \rightarrow \infty}\{f(x)-v^g(x)\} \; \leq \; f^g(\theta).
\end{eqnarray*}
Hence $v(\theta)=f^g(\theta)$ on $\Theta$.
\end{proof}

\begin{remark}
Case 1 of the Theorem gives the sufficient condition mentioned in the 
paragraph headed Inverse Problem in 
Section~\ref{s:tfatip}. If $\theta_- \in \Theta$ then $\xminus = 0$ if 
and 
only if for all $x>0$, $g(x,\theta_-) - v^g(x) < g(0, 
\theta_-)$, where we use the fact that, by supposition, 
$v^g(0)=0$. 
\end{remark}

\subsection{Non-Uniqueness}
Given existence of a diffusion $X$ which is consistent with the values 
$V(\theta)$, the aim of the next few sections is to determine whether 
such a diffusion is unique.

Fundamentally, there are two natural ways in which uniqueness 
may fail. Firstly, the domain $\Theta$ may be too small (in extreme 
cases $\Theta$ might contain a single element). Roughly speaking the 
$g$-convex duality is only sufficient to determine $v^g$ (and hence the 
candidate $\phi$) over $(\xminus,\xplus)$ and there can be many 
different convex extensions of $\phi$ to the real line, for each of 
which $\psi^g = v$. 
Secondly, even when $\xminus=0$ and $\xplus=\infty$, 
if $\xstar(\theta)$ is discontinuous then there can be
circumstances in which there are a multitude of convex functions 
$\phi$ with $(\log \phi)^g = v$. In that case, if there are no $\theta$ 
for which it is optimal to stop in an interval $I$, then it is only 
possible to infer a limited amount about the speed measure of the 
diffusion over that 
interval.

In the following lemma we do not assume that $\psi$ is $g$-convex.

\begin{lemma}
\label{l:psig=v}
Suppose $v$ is $g$-convex and $\psi^g=v$ on $\Theta$. 
Let $A(\theta) = \{x: g(x,\theta) - \psi(x) = \psi^g(\theta) \}$.
Then, for each
$\theta$,
$A(\theta) \subseteq \partial^g \psi^g (\theta) \equiv \partial^g v 
(\theta)$, and for
$x \in A (\theta)$, $\psi(x)=\psi^{gg}(x) =v^g(x)$. Further, for
$\theta \in (\theta_-,\theta_R)$ we have $A(\theta) \neq \emptyset$. 
\end{lemma}

\begin{proof}
Note that if $\psi$ is any function,
with $\psi^g = v$ then
$\psi \geq \psi^{gg} = v^g$.

If $\hx \in A(\theta)$ then
\[ \psi^g(\theta) = g(\hx,\theta)- \psi(\hx) \leq g(\hx,\theta)-
v^g(\hx) \leq v(\theta). \]
Hence there is equality throughout, so $\hx \in \partial^g v
(\theta) $ and $\psi(\hx)= v^g(\hx) = \psi^{gg}(\hx)$.

For the final part, suppose $\theta < \theta_R$ and fix
$\tilde{\theta} \in
(\theta,\theta_R)$.
From the Spence-Mirrlees condition, if $x>\tilde{x}:= 
\xstar(\tilde{\theta})$,
\[ g(x,\theta) - g (\tilde{x},\theta) <
        g(x,\tilde{\theta}) - g (\tilde{x},\tilde{\theta}), \]
and hence
\[ \{ g(x,\theta) - \psi(x) \} - \{ g (\tilde{x},\theta) - 
\psi(\tilde{x}) \} 
< \{  g(x,\tilde{\theta}) - \psi(x) \} - \{ g 
(\tilde{x},\tilde{\theta}) - \psi(\tilde{x}) \} \leq 0.
\]
In particular, for $x>\tilde{x}$, $g(x,\theta) - \psi(x) < g 
(\tilde{x},\theta) - \psi(\tilde{x} )$ and
\[ \sup_{x \geq 0} g(x,\theta) - \psi(x) =
\sup_{0 \leq x \leq \tilde{x}} g(x,\theta) - \psi(x) . \]
This last supremum is attained so that $A(\theta)$
is non-empty.

\end{proof}

\subsection{Left extensions}
\label{SS:leftextension}

In the case where $\xminus>0$ and there exists a diffusion consistent 
with $V$ then it is generally possible to construct many diffusions 
consistent with $V$. Typically $V$ contains insufficient information to 
characterise the behaviour of the diffusion near zero.

Suppose that $0<\xminus<\infty$. Recall the definition of the 
straight line $F$ from 
Theorem~\ref{t:suffexist}.

\begin{lemma} \label{l:extendunique}
Suppose that $0<\xminus<\infty$ and that there
exists $X \in \sX_0$ consistent with $V$.

Suppose that $\theta_R>\theta_-$ and that $v^g$ and is 
continuous and differentiable to the right at $\xminus$. Suppose further 
that $\xstar(\theta) > \xminus$ for each $\theta>\theta_-$.

Then, unless either ${v^g(x)}=f(x)$ for some $x \in [0,\xminus)$
or $(v^g)'(\xminus+)=f'(\xminus)$, there are many diffusions 
consistent with $V$.
\end{lemma}

\begin{proof}
Let $\phi$ be the log-eigenfunction of a diffusion $X \in \sX_0$ which 
is consistent with $V$ 

If $\theta_- \in \Theta$ then
$v^g(\xminus) = \psi(\xminus)$ by Lemma~\ref{l:psig=v}.
Otherwise the same conclusion holds on taking limits, since the 
convexity of $\phi$ necessitates continuity of $\psi$.

Moreover, taking a sequence $\theta_n \downarrow \theta_-$, and using
$\hat{x}(\theta_n) > \xstar(\theta_n -) > \xminus$
we have
\[ \psi'(\xminus+) = \lim_{\theta_n \downarrow \theta_-} 
\frac{\psi(\hx(\theta_n)) - \psi(\xminus)}{\hx(\theta_n) -\xminus}
= \lim_{\theta_n \downarrow \theta_-}
\frac{v^g(\hx(\theta_n)) - v^g(\xminus)}{\hx(\theta_n) -\xminus}
= (v^g)'(\xminus+) \]

In particular, the conditions on $v^g$ 
translate directly into conditions about $\phi$.

Since  the straight line $F$ is
the largest convex function with
$F(0)=1$ and $F(\xminus) = e^{v^g(\xminus)}$ we must have $\phi \leq F$.

Then if $\phi(x)=F(x)$ for some $x \in (0,\xminus)$
or $\phi'(\xminus+)=F'(\xminus)$, then convexity of $\phi$ guarantees 
$\phi=F$ on $[0,\xminus]$.

%{\bf Need to argue $(v^g)'(xminus-) \leq (v^g)'(xminus+)$}

Otherwise there is a family of convex, increasing $\tilde{\phi}$ with
$\tilde{\phi}(0)=1$ and such that 
$v^g(x) \leq 
\log \tilde{\phi}(x) \leq F(x)$ for $x<\xminus$ and 
$\tilde{\phi}(x)=\phi(x)$ for $x \geq \xminus$.

For such a $\tilde{\phi}$, then by the arguments of Case 2 of 
Theorem~\ref{t:existence} we have $(\log \phi_F)^g = v$ and then
$v^g \leq \log \tilde{\phi} \leq \phi_F$ implies
$v \geq (\log \tilde{\phi})^g \geq
(\log \phi_F)^g = v$.

Hence each of $\tilde{\phi}$ is the eigenfunction of a diffusion which 
is consistent with $V$.
\end{proof}

\begin{example}
Recall Example \ref{ex:powers}, in which we have $\xminus = 1$, 
$\varphi'(1)=2$ 
and 
$\varphi(1)=2$.
We can extend $\varphi$ to $x \in [0,1)$ by (for example)
drawing the straight line between $(0,1)$ and
$(1,2)$ (so that for $x \leq 1$, $\varphi(x)=1+x$). With this choice
the resulting extended function will be convex, thus defining a 
consistent
diffusion on $\R^+$. Note that any convex extension of 
$\varphi$ (i.e. any function $\hat{\varphi}$
such that $\hat{\varphi}(0)=1$ and $\hat{\varphi}'(0-)=0$,
$\hat{\varphi}(x)=\varphi(x)$ for $x>1$) solves the inverse problem, 
(since necessarily $\hat{\varphi}(x) \geq 2x = e^{v^g(x)}$ on $(0,1)$).  
The most natural choice is, perhaps, $\varphi(x)=1+x^2$ for $x \in (0,1)$.
\end{example}

Our next lemma covers the degenerate case where there is no optimal 
stopping rule, and for all $\theta$ it is never optimal to stop. 
Nevertheless, as Example~\ref{e:tanh} below shows, the theory of 
$u$-convexity as developed in the article still applies.

\begin{lemma}
Suppose $\xminus=\infty$, and that there exists a 
convex increasing function $F$ with $F(0)=1$ and such that
$\log F(x) \geq v^g(x)$ and $\lim_{x \rightarrow \infty} \{ \log F(x) - 
v^g(x) \}=0$.

Suppose that $\lim_{x \rightarrow \infty}
{e^{v^g(x)}}/{x}$ exists in $(0,\infty]$ and write
$\kappa = \lim_{x \rightarrow \infty}
{e^{v^g(x)}}/{x}$.
If $\kappa
<\infty$ then $X$ is the unique diffusion consistent with $V$ 
if and only if $e^{v^g(x')} = 1+\kappa x'$ for some $x' > 0$
or $\lim \sup_{x \uparrow \infty} (1 + \kappa x) - e^{v^g(x)}=0$. 
If  $\kappa=\infty$ then there exist uncountably many diffusions 
consistent with $V$.
\end{lemma}

\begin{proof} 
The first case follows similar reasoning as Lemma \ref{l:extendunique} above. 
Note that $x \mapsto  1+\kappa x$ is the largest convex function $F$ on 
$[0,\infty)$ such that $F(0)=1$ and $\lim_{x \rightarrow 
\infty}\frac{F(x)}{x}=\kappa$. 

If $e^{v^g(x')}=1+\kappa x'$ for any $x'> 0$, or if 
$\lim \sup_{x \uparrow \infty} (1 + \kappa x) - e^{v^g(x)}=0$
then 
there does not exist any convex function lying between $1+\kappa x$ 
and $e^{v^g(x)}$ on $[0,\infty)$.
In particular $\phi(x) = 1+\kappa x$ 
is the unique eigenfunction consistent with $V$. 

Conversely, if 
$e^{v^g}$ lies strictly below the straight line $1+\kappa x$, and if  
$\lim \sup_{x \uparrow \infty} (1 + \kappa x) - e^{v^g(x)}>0$ then it is 
easy to 
verify that we can find other increasing convex functions with initial 
value 1, satisfying the same limit 
condition and lying between $e^{v^g}$ and the line. 

In the second case define 
$F_\alpha(x)=F(x)+\alpha x$ for $\alpha >0$. Then since 
$\lim_{x \rightarrow \infty} e^{v^g(x)}/{x}=\infty$ we have
\[\lim_{x \rightarrow \infty} \frac{F_\alpha(x)}{e^{v^g(x)}} 
= \frac{F(x)}{e^{v^g(x)}} = 1\]

Hence $F_\alpha$ is the eigenfunction of another diffusion which is 
consistent with $V$.  We conclude
 that there exist uncountably many consistent diffusions.
\end{proof}

\begin{example} 
\label{e:tanh}
Suppose $g(x,\theta)=x^2+\theta \tanh x$ and 
$v(\theta)=\theta$ on $\Theta = \R^+$. 
For this example we have that $v$ is nowhere $g$-subdifferentiable and 
$\xminus=\infty$.
Then $v^g(x)=x^2$ and each of $\varphi(x)=e^{x^2}$,
\[\bar{\varphi}(x) = \left\{ \begin{array}{ll} 
1+(e-1)x & 0 \leq x < 1 \\
e^{x^2} & 1 \leq x,  
\end{array} \right. \]
and $\varphi_\alpha(x)=\varphi(x)+ \alpha x$ for any $\alpha \in 
\R_+$ is an eigenfunction consistent with $V$. 
\end{example}

\subsection{Right extensions}
\label{SS:rightextension}

The case of $\xplus<\infty$ is very similar to the case $\xminus>0$, and 
typically if there exists one diffusion $X \in \sX_0$ which is 
consistent with $V$, 
then there exist many such diffusions. Given $X$ consistent with 
$V$, the idea is to produce 
modifications of the eigenfunction $\varphi_X$ which agree with 
$\varphi_X$ on 
$[0,\xplus]$, but 
which are different on $(\xplus,\infty)$.

\begin{lemma}
\label{l:extendR}
Suppose $\xplus<\infty$. Suppose there exists 
a diffusion $X \in \sX_0$ such that $V_X$ agrees with $V$ on $\Theta$.
%$\hat\phi$ on $[0,\xplus)$ 
%with $\hat\phi(0)=1$ and $\hat\phi$ increasing and convex, and $(\log 
%\hat\phi)^g = v$ on $\Theta$. 
If $v^g(\xplus) + (v^g)'(\xplus+) < 
\infty$ then there are infinitely many 
diffusions in $\sX_0$ which are consistent with $V$. 

\end{lemma}

\begin{proof} 
It is sufficient to prove that given convex increasing $\hat\phi$ 
defined on 
$[0,\xplus)$
with $\hat\phi(0)=1$ and $(\log \hat\phi)^g = v$ on $\Theta$, then there 
are many increasing, convex $\phi$ with
defined on $[0,\infty)$ with $\phi(0)=1$ for which $(\log \phi)^g=v$.

The proof is similar to that of Lemma~\ref{l:extendunique}.
\end{proof}

\begin{example}\label{ex:david}
Let $G(x,\theta)=%\frac
{\theta x}/{(\theta + x)}$, and 
$\Theta=(0,\infty)$. 

Consider the forward problem when $X$ is a 
reflecting Brownian motion, so that the eigenfunction is given by 
$\varphi(x) =  \cosh (x \sqrt{2 \rho})$. Suppose $\rho=1/2$.

Then $\{g(x,\theta)-\log(\cosh x)\}$ attains its maximum at the solution 
$x=\xstar(\theta)$ to
\begin{equation}\label{eq:davids}
\theta = \frac{x^2 \tanh x}{1 - x \tanh x}.
\end{equation}
It follows that $\xminus=0$ but $\xplus = \lim_{\theta \uparrow \infty}
\xstar(\theta) = \hat{\lambda}$ where $\hat\lambda$ is the positive root 
of ${\mathcal L}(\lambda)=0$ and
${\mathcal L}(\lambda) = 1 - \lambda  \tanh \lambda$.

Now consider an inverse problem.
Let $G$ and $\Theta$ be as above, and suppose $\rho=1/2$.
Let $\xstar(\theta)$ be the solution to (\ref{eq:davids}) and let 
$v(\theta) = g(\xstar(\theta),\theta)-\log(\cosh \xstar(\theta))$.
Then the diffusion with speed measure $m(dx)=dx$ (reflecting Brownian 
motion), is an element of $\sX_0$ which is consistent with 
$\{V(\theta) : \theta \in (0,\infty) \}$. However, this solution of the 
inverse problem is not 
unique, and any convex function $\varphi$ with $\varphi(x) = \cosh x$ 
for $x \leq \hat{\lambda}$ is the eigenfunction of a consistent 
diffusion. To see this note that for $x > \xplus$, $v^g(x) = 
\lim_{\theta \uparrow \infty} \{ g(x,\theta) - v(\theta) \} = \log (x 
\cosh (\xplus)/\xplus)$ so that any convex $\varphi$ with $\varphi(x)= 
\cosh x$ for $x \leq \xplus$ satisfies $\varphi \geq e^{v^g}$.
\end{example}

\begin{remark}
If $\xplus + v^g(\xplus) + (v^g)'(\xplus+) < \infty$ then one 
admissible 
choice is to take 
$\phi=\infty$. This was the implicit choice in the proof of 
Theorem~\ref{t:existence}.
\end{remark}

\begin{example}
The following example is `dual' to Example~\ref{ex:BM1a}.

Suppose $\rho=1/2$, $g(x,\theta)= \theta x$, $\Theta = (0,\infty)$ and $v(\theta)= 
\log (\cosh \theta)$. Then $v^g(x) = x \tanh^{-1}(x) + \frac{1}{2} 
\log (1 - x^2)$, for $x\leq 1$. For $x>1$ we have that $v^g$ is 
infinite. Since $v$ is convex, and $g$-duality is 
convex 
duality, we conclude that $v$ is $g$-convex. Moreover,
$v^g$ is convex. Setting $\psi = v^g$ we have that $\psi(0)=0$, $\varphi 
= e^\psi$ is convex and $\psi^g = v^{gg}=v$. Hence $\psi$ is associated 
with a diffusion consistent with $V$, and this diffusion has an 
absorbing boundary  
at $\xi \equiv 1$.

For this example we have $\xplus=1$ and $v^g(\xplus)=\log 2$, but the 
left-derivative of $v^g$ is infinite at $\xplus$ and $v^g$ is infinite 
to the right of $\xplus$. Thus there is 
a unique 
diffusion in $\sX_0$ which is consistent with $V$.
\end{example}

\subsection{Non-Uniqueness on $[\xminus,\xplus)$}
\label{SS:middleextension}

Even if $[\xminus,\xplus)$ is the positive real line, then if 
$x^*(\theta)$ fails to be continuous it is possible that there are 
multiple diffusions consistent with $V$.

\begin{lemma} \label{l:uniqueness} 
Suppose there exists a diffusion $X \in \sX_0$ which is consistent with 
$\{V(\theta):\theta \in \Theta \}$.

Suppose the $g$-subdifferential of $v$ is multivalued, or more generally 
that $x^*(\theta)$ is not continuous on $\Theta$. Then
there exists an interval $I \subset (\xminus,\xplus)$ where the 
$g$-subdifferential 
of $\psi = v^g$ is constant, so that $\theta^*(x) = \bar{\theta}, \; 
\forall x \in I$. 
If $G(x, \bar{\theta}) = e^{g(x, \bar{\theta})}$ is strictly convex in 
$x$ on some subinterval of $I_0$ of $I$ then the diffusion $X$
is not the unique element of  
$\sX_0$ which is consistent with $V$.
\end{lemma}

\begin{proof}
First note that if $\xstar(\theta)$, is 
continuous then $\thstar=\xstar^{-1}$ is nowhere constant  and hence 
strictly monotone and thus $\psi$=$v^g$ is strictly $g$-convex (recall \ref{d:310}).

Suppose $\bar{G}(x) := G(x, \bart)$ is strictly convex on $I_0  
\subseteq I$. Then we can choose $\hat{G}$ such that \\
$\bullet$ $\hat{G} = \bar{G}$ on $I_0^c$, \\
$\bullet$ $\hat{G}$ is linear on $I_0$, \\
$\bullet$ $\hat{G}$ is continuous. \\
Then $\hat{G}(x) \geq G(x, \bart)$.

By definition we have 
\[ \psi(x) = g(x, \bart) - \psi^g(\bart) \hspace{10mm} x \in I . \]
Then $\varphi_X(x) = G(x,\bart)/V(\bart)$ on $I$, see Figure 2. 

Let $\hat{\varphi}$ be given by
\[
\hat{\varphi}(x) = \left\{ \begin{array}{ll}
        \varphi_X(x) & \mbox{ on $I_0^c$ }, \\
        \frac{\hat{G}(x)}{V(\bart)} &  \mbox{ on $I_0$ } 
            \end{array} \right. \]
Then $\hat{\varphi}$ is convex and $\hat{\varphi} \geq \varphi$, so that 
they are associated with different elements of $\sX_0$. Let $\hat{\psi} = 
\ln \hat{\varphi}$.

It remains to show that $\hat{v}:=\hat{\psi}^g = \psi^g=v$. It is 
immediate from $\hat{\psi} \geq \psi$
that $\hat{\psi}^g \leq \psi^g$. For the converse, observe that

\begin{eqnarray*}
v(\theta) & = & \left( \sup_{x \in I_0} \{ g(x,\theta) - \psi(x) \} 
\right) \vee \left( \sup_{x \in I_0^c} \{ g(x,\theta) - \psi(x) \} 
\right) \\
& = & \sup_{x \in I_0^c} \{ g(x,\theta) - \psi(x) \}  \\
& = & \sup_{x \in I_0^c} \{ g(x,\theta) - \hat{\psi}(x) \}  \\
& \leq & \hat{v}(\theta).
\end{eqnarray*}

\end{proof}

\begin{figure}[t]\label{Fig.2}
\begin{center}
\includegraphics[height=8cm,width=8cm]{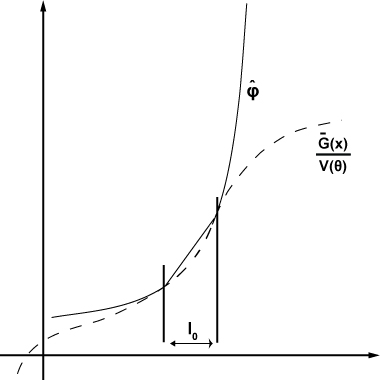}
\caption{The dashed line represents a function 
$e^{g(x,\theta)+c}$ for some $\theta \in \Theta$ and some $c$. $\varphi$
is given by the solid line on $I_0^c$ and the dashed line on $I_0$. Then 
$\psi = \log \phi$ 
has a $g$-section 
over $I_0$ and $G$ is convex there. $\hat{\varphi}$, given by the solid 
line is another 
eigenfunction consistent with $V$.}
\end{center}
\end{figure}

\begin{example}
Suppose $G(x,\theta)=e^{\theta x}$, and $\Theta = (0,\infty)$. Suppose 
that $X$ is such that $\psi$ 
is given by 
\[ \psi(x) = \left\{ \begin{array}{ll} 
\frac{x^2}{4} & x < 2 \\
(x-1)         & 2 \leq x < 3 \\
\frac{x^2}{6} + \frac{1}{2} & 3 \leq x
\end{array} \right. \]
It follows that
\[ v(\theta) = \left\{ \begin{array}{ll}
{\theta^2} & \theta \leq 1 \\
\frac{3 \theta^2}{2} - \frac{1}{2} & 1 < \theta
\end{array} \right. \]

Then $\partial^g v(1)$ is multivalued, and there are a family of 
diffusions $\tilde{X} \in \sX_0$ which give the same value functions as 
$X$.

In particular we can take
\[ \hat{\psi}(x) = \left\{ \begin{array}{ll}
\frac{x^2}{4} & x < 2 \\
\log ((e^2-e)x + 3e - 2e^2)         & 2 \leq x < 3 \\
\frac{x^2}{6} + \frac{1}{2} & 3 \leq x
\end{array} \right. \]
Then $\hat{\psi}^g = v(\theta)$ and $\hat{\psi}$ is a log-eigenfunction.

\end{example}

\subsection{Uniqueness of diffusions consistent with V} 
\label{SS:uniqueness}

\begin{proposition} \label{p:totalunique} 
Suppose $V$ is such that $\xstar(\theta)$ is continuous on $\Theta$, 
with range the positive real line.

Then there exists at most one diffusion $X \in \sX_0$ consistent with 
$V$.
\end{proposition}

\begin{proof}
%If $\xstar(\theta)$ is continuous and strictly increasing, and if
%the range of $x^*(\theta)$ is the whole of the positive real line, then
The idea is to show that
${v^g(x)}$ is the only function with $g$-convex dual $v$.
Suppose $\psi$ is such that $\psi^g=v$ on $\Theta$.
For each $x$ there is a $\theta$ with $\xstar(\theta)=x$, and
moreover $\partial^g v(\theta) = \{x \}$. Then by Lemma~\ref{l:psig=v},
$A(\theta) = \{x \}$ and $\psi(x)= v^g(x)$.
\end{proof}

Recall that we define $\theta_R = \sup_{\theta} \partial^g 
v(\theta) \neq \emptyset$ and if $\theta_R>\theta_-$ set $x_R =
\lim_{\theta \uparrow \theta_R} \xstar(\theta)$.

\begin{theorem} \label{t:totalunique}
Suppose $V$ is such that $v$ is continuously differentiable on 
$(\theta_-,\theta_R)$ and that $\xminus=0$ and $x_R = \infty$.

Then there exists at most one diffusion $X \in \sX_0$ consistent with
$V$.
\end{theorem}

\begin{proof}
The condition on the range of $\xstar(\theta)$ translates into the 
conditions on $\xminus$ and $x_R$, so
it is sufficient to show that $\xstar(\theta)$ is continuous at 
$\theta \in (\theta_-,\theta_R)$ if and only if $v$ is differentiable 
there. This follows from Proposition~\ref{p:prop2}.
\end{proof}

\begin{corollary}
If the conditions of the Theorem hold but either $e^{v^g(x)}$
is not convex or $v^g(0) \neq 0$, then there is no
diffusion $X \in \sX_0$ which is consistent with $\{ V(\theta), \theta
\in \Theta \}$.
\end{corollary}

\begin{example}
\label{ex:concave2}
Recall Example~\ref{ex:concave}. For this example we have
$\xstar(\theta)= \theta^2 -1$,
which on $\Theta=(1,\infty)$ is continuous and strictly increasing.
Then $e^{v^g(x)} = \sqrt{1+x}$ and by the above corollary there is no
diffusion consistent with $v$.
\end{example}

\begin{remark}
More general but less succinct sufficient conditions for uniqueness can 
be deduced from Lemma~\ref{l:extendunique} or Lemma~\ref{l:extendR}. 
For 
example, 
if $0<\xminus<\xplus=\infty$, but $(v^g)'(\xminus)= (1 - 
e^{-v^g(\xminus)})/\xminus$ then there 
is 
at most one $X \in \sX_0$ which is consistent with $V$.
\end{remark}

\section{Further examples and remarks}

\subsection{Birth-Death processes}
\label{S:bd}

We now return to $\sX_0$ and consider the case when $E$ is made up of
isolated points only; whence
$X$ is a birth-death process on points $x_n \in E$ indexed by $n \in
\N_0$,
with associated exponential holding times $\lambda_n$.
We assume $x_0=0$, $x_n$ is increasing, and write $x_\infty = \lim_n
x_n$.

For a birth-death process the transition probabilities are given by
\begin{eqnarray*} %\label{eq:transitionbd}
\Prob_{n,n+1}(t)=p_n \lambda_n t + o(t), \\
\Prob_{n,n-1}(t)=q_n \lambda_n t + o(t),
\end{eqnarray*}
where of course $q_n=1-p_n$, with $p_0=1$.
By our assumption that, away from zero, $(X_t)_{t \geq 0}$ is a martingale,
we must have $p_n=\frac{x_{n+1}-x_n}{x_{n+1}-x_{n-1}}$. Then we can
write
$x_n=x_{n-1}+\frac{\prod_{i=1}^{n-1} q_i}{\prod_{i=0}^{n-1}
p_i}$.
Let \[m(x_n)=\frac{1}{\lambda_n} \frac{p_0
p_1 ... p_{n-1}}{q_1 q_2 ... q_n} .\]
Then it is easy to verify, but see
\cite{feller}, that (\ref{eq:differential}) can be expressed in terms of
a second-order difference operator
\begin{equation} \label{eq:birthdeathdiff}
\frac{1}{m(x_n)}\left[\frac{f(x_{n+1})-f(x_{n})}{x_{n+1}-x_n}
- \frac{f(x_{n})-f(x_{n-1})}{x_{n}-x_{n-1}} \right] - \rho f(x_n)=0,
\end{equation}
with boundary conditions $f(0)=1$ and $f'(0-)=0$.

Let $M(x)=\sum_{x_n < x} m(x)$. In the language of strings, the pair
$(M,[0,x_\infty))$ is known as the Stieltjes String. If $x_{\infty} +
M(x_{\infty}) < \infty$ the string is called regular and $x_{\infty}$ is
a regular boundary point, while otherwise the string is called singular,
in which case we assume that $x_{\infty}$ is natural (see
Kac~\cite{kac}).

In  this section we consider the call option payoff,
$G(x,\theta)=(x-\theta)^+$ defined for $\theta \in \Theta =
[\theta_0,\infty)$. This objective function is straight-forward to
analyse
since the $g$-duality corresponds to straight lines in the original
coordinates.
It follows that for the forward problem $V$ is decreasing and convex in
$\theta$. $V$ is easily seen to be piecewise linear.

Our focus is on the inverse problem. Note that the solution of this 
problem involves finding the space $E$ and the jump rates $\lambda_n$. 
Suppose that $V$ is decreasing, convex and piecewise linear. Let 
$(\theta_n)_{n \in \N_0}$ be a sequence of increasing real valued 
parameters with $\theta_0 < 0$ and $\theta_n$ increasing to infinity, 
and suppose that $V$ has negative slope $s_i$ on each interval 
$(\theta_i,\theta_{i+1})$. Then $s_i$ is increasing in $i$ and
\begin{equation} \label{eq:bdV}
V(\theta)=V(\theta_n) + (\theta-\theta_n) s_n \hspace{10mm} \mbox{for $
\theta_n \leq \theta < \theta_{n+1}$}.
\end{equation}
%If $\theta_n \rightarrow \infty$ as $n \rightarrow \infty$ then we
%require that $V(\theta_n) \rightarrow 0$.
We assume that
$s_0=\frac{\theta_0}{V(\theta_0)} < 0$.

Since $V$ is convex, $v$ is $\log((x-\theta)^+)$ convex. Let
$\varphi(x)=\exp(v^g(x))$.
% and set
%$\xstar(\theta_0)=\theta_0-\frac{1}{v'(\theta_0)}$, $\xi=\lim_{n
%\rightarrow \infty} \xstar(\theta_n)$.
By Proposition \ref{p:prop1}, for $\theta \in [\theta_n,\theta_{n+1})$
\[
\frac{-1}{\xstar(\theta)-\theta} =
g_\theta(\xstar(\theta),\theta)
=\frac{s_n}{V(\theta_n)+(\theta-\theta_n)s_n}
\]
so that
$x_n := \xstar(\theta_n)= \theta_n-V(\theta_n)/s_n$.
Note that $\xstar(\theta)$ is constant on $[\theta_n,\theta_{n+1})$.
We find that for $\theta \in [\theta_n,\theta_{n+1})$
\[
\psi(\xstar(\theta)) = \log(\theta_n-\theta-V(\theta_n)/s_n)-v(\theta),
\]
and hence $\varphi(\xstar(\theta)) = \frac{-1}{s_n}$.
Then, for $x \in [\xstar(\theta_n),\xstar(\theta_{n+1}))$,
\begin{equation}
\psi(x)=\log(x-\theta_n)-v(\theta_n) .
\end{equation}

We proceed by determining the {\it $Q$}-matrix for the process on
$[\xstar(0)=0,\xi)$. For each $n$, let $p_n$
denote the probability of jumping to state ${x_{n+1}}$ and $q_n$ the
probability of jumping to ${x_{n-1}}$. Then $p_n$ and $q_n$ are
determined by the martingale property (and $p_0=1$).
Further $\lambda_n$
is determined either through (\ref{eq:birthdeathdiff}) or from a
standard
recurrence relation for first
hitting
times of birth-death processes:
\begin{equation*}
\lambda_n=\frac{\rho \varphi(x_n)}
{p_n \varphi(x_{n+1}) + (1-p_n) \varphi(x_{n-1})-\varphi(x_n)},
\hspace{15mm} n \geq 1.
\end{equation*}

\begin{example} \label{eg:bd}
Suppose that $\theta_n = n + 2^{-n}-2$ so that $\theta_0=-1$, and
$V(\theta_n) = 2^{-n}$. It follows that $s_n = -(2^{n+1}-1)^{-1}$.
We find $x_n = n$ (this example has been crafted to ensure
that the birth-death process has the integers as state space, and this 
is not a general result).
Also $\varphi(n) = 2^{n+1} - 1$ ($\varphi$ is piecewise
linear with kinks at the integers) and the holding time at $x_n$ is 
exponential with rate
$\lambda_n =4 \rho ( 1 -
2^{-(n+1)})$.
\end{example}

\subsection{Subsets of $\sX$ and uniqueness}
\label{ss:otherX}

So far in this article we have concentrated on the class $\sX_0$. 
However, the
methods and ideas translate to other classes of diffusions.

Let $\sX_{m,s}$ denote the set of all diffusions
reflected at $0$. Here $m$ denotes the speed measure, and $s$ the
scale function. With the
boundary conditions as in (\ref{eq:differential}), $\varphi(x) \equiv
\varphi_X(x)$ is the increasing, but not necessarily convex solution to
\begin{equation} \label{eq:differentialsc}
\frac{1}{2} \frac{d^2 f}{dm ds} = \rho f .
\end{equation}

In the smooth case, when $m$ has a density $m(dx)= \nu(x)dx$ and $s''$
is continuous,
(\ref{eq:differentialsc}) is equivalent to
\begin{equation} \label{eq:differentialsnice}
\frac{1}{2} \sigma^2(x) f''(x) + \mu(x) f'(x) = \rho f(x),
\end{equation}
where \[ \nu (x)=  \sigma^{-2}(x) e^{M(x)}, \ \ s'(x)=e^{-M(x)}, \ \
M(x)=\int_{0-}^x 2 \sigma^{-2} (z) \mu(z) dz, \]
see \cite{borodin}.

Now suppose $V \equiv \{ V(\theta) : \theta \in \Theta \}$ is given 
such
that $v^u(0)=0$, $(v^u)'(0)=0$ and $v^u$ is increasing, then we will be
able to find several pairs $(\sigma,\mu)$ such that $\exp(v^u)$ solves
(\ref{eq:differentialsnice}) so that there is a family of diffusions
rather than a unique diffusion in $\sX_{m,s}$ consistent with $v$.

It is only by considering subsets of $\sX_{m,s}$, such as taking
$s(x)=0$ as in the majority of this article, or perhaps by setting
the diffusion co-efficient equal to unity, that we can hope to find a
{\em unique} solution to the inverse problem.

\begin{example} Consider Example $\ref{ex:concave}$ where we found
$\psi(x)=\sqrt{1+x}$. Let $\sX_{1,s}$ be the set of diffusions with
unit
variance and scale function $s$ (which are reflected at 0). Then there
exists a unique diffusion in
$\sX_{1,s}$ consistent with $V$. The drift is given by
\[\mu(x)=\frac{1/4+2 \rho (1+x)^2}{1+x}.\]
\end{example}

\section{Applications to Finance}

\subsection{Applications to Finance}
\label{SS:finance}
Let $\sX_{stock}$ be the set of diffusions with the representation
\[dX_t = (\rho-\delta) X_t dt + \eta(X_t) X_t dW_t.\]
In finance this SDE is often used to model a stock price process, with 
the 
interpretation that $\rho$ is the interest rate, $\delta$ is the 
proportional dividend, and $\eta$ is the level dependent volatility. 
Let
$x_0$ denote the starting level of the diffusion and suppose that $0$ is an
absorbing barrier.

Our aspiration is to recover the underlying model, assumed to be an 
element of $\sX_{stock}$, given a set of perpetual American option 
prices, parameterised by $\theta$. The canonical example is when 
$\theta$ is the strike, and $G(x,\theta)=(\theta-x)^+$, and then,
as discussed in Section~\ref{ss:otherX}, the fundamental ideas 
pass over from $\sX_0$ to $\sX_{stock}$. We suppose $\rho$ and $\delta$ 
are given and aim to recover the volatility $\eta$.

Let $\varphi$ be the convex and decreasing solution to the differential 
equation
\begin{equation} \label{eq:eigenstock}
\frac{1}{2} \eta(x)^2 x^2 f_{xx} + (\rho-\delta) x f_x - \rho f = 0.
\end{equation}
(The fact that we now work with decreasing $\varphi$ does not invalidate 
the method, though it is now appropriate to use payoffs $G$ which are 
monotonic decreasing in $x$.)
Then $\eta$ is determined by the Black-Scholes equation
\begin{equation}
\label{e:BSdef}
\eta(x)^2=2 \frac{ \rho \varphi(x)- (\rho-\delta) x \varphi 
'(x)}{x^2 \varphi ''(x)} . 
\end{equation}

Let $G \equiv G(x,\theta)$ be a family of payoff functions satisfying assumption \ref{ass}.
Under the additional assumption that $G$ is decreasing in $x$ (for example, the put payoff) Lemma
\ref{l:coincide} shows that the optimal stopping problem reduces to 
searching over
first hitting times of levels $x < x_0$.
Suppose that $\{V(\theta) ; \theta \in \Theta\}$ is used to determine a 
smooth, 
convex
$\varphi=\exp(\psi)$ on $[0,\xi)$ via
the $g$-convex duality
\[\psi(x)=v^g(x)=\sup_{\theta \in \Theta} [g(x,\theta)-v(\theta)].\]
Then the inverse problem is solved by the diffusion with 
volatility 
given by the solution of (\ref{e:BSdef})  
above. 
Similarly, given a diffusion $X \in 
\sX_{stock}$
such that $\psi=\log(\varphi)$ is $g$-convex on $[0,\xi)$, then the 
value 
function
for the optimal stopping problem is given exactly as in Proposition 4.4.
See Ekstr\"om and Hobson~\cite{hobson} for more details.

\begin{remark} \label{r:stockirregular} The irregular case of a 
generalised diffusion  
requires the
introduction of a scale function, see Ekstr\"{o}m and Hobson
\cite{hobson}. That article restricts itself to the case of the put/call 
payoffs
for diffusions $\sX_{stock}$, using arguments from regular convex analysis 
to
develop the duality relation. However, the construction of the scale 
function is
independent of the form of the payoff function and is wholly transferable 
to the setting of $g$-convexity used in this paper. \end{remark}

\subsection{The Put-Call Duality; $g$-dual Processes}
\label{SS:aj}

In \cite{alfonsi3}, Alfonsi and Jourdain successfully obtain what they 
term 
an American Put-Call duality (see below) and in doing so, they solve 
forward and inverse problems for diffusions $X \in \sX_{stock}$. 
They consider objective functions corresponding to 
calls and puts: 
$G(x,\theta)=(x-\theta)^+$ 
and $G(x,\theta)=(\theta-x)^+$.
In 
\cite{alfonsi2} the procedure is generalised slightly to payoff functions 
sharing `global properties' with the put/call payoff. In our notation the 
assumptions they 
make are the following: 
\cite[(1.4)]{alfonsi2} \\
Let $G: \R^+ \times \R^+ \rightarrow \R^+$ be a continuous function such 
that on
the space $\Phi = \{ (x,\theta) : G(x,\theta) > 0 \} \neq \emptyset$, $G$ 
is $C^2$
and further for all $x,\theta \in \Phi$
\[ G_x(x,\theta) < 0, \hspace{5mm} G_\theta(x,\theta) > 0,
\hspace{5mm} 
G_{xx}(x,\theta) \leq 0
\hspace{5mm}
G_{\theta \theta}(x,\theta) \leq 0.
\]
Subsequently, they assume 
\cite[(3.4)]{alfonsi2}
\begin{equation} \label{eq:expsm}
G G_{x\theta} > G_x G_\theta \hspace{5mm} \mbox{on $\Phi$.}
\end{equation}
Condition (\ref{eq:expsm}) is precisely the Spence-Mirrlees condition 
applied 
to $g=\log(G)$. Note that unlike Alfonsi and Jourdain~\cite{alfonsi2} we 
make no concavity 
assumptions on G. We also treat the case of the reverse Spence-Mirrlees
condition, and we consider classes of diffusions other than $\sX_{stock}$. Further, 
as in
Ekstr\"{o}m and Hobson~\cite{hobson} we allow for generalised 
diffusions.
This is important if we are to construct solutions to the inverse 
problem when 
the value functions are not smooth. Moreover, even when $v$ is smooth, if 
it contains a $g$-section, the diffusion which is consistent with $v$ 
exists only in the generalised sense. When $v$ contains a $g$-section we 
are able to address the question of uniqueness. Uniqueness is automatic 
under 
the additional monotonicity assumptions of \cite{alfonsi2}. 

In addition to the solution of the inverse problem, a further aim of
Alfonsi and Jourdain~\cite{alfonsi3} is to 
construct dual pairs of models such that call prices (thought of as a 
function of strike) under one model become put prices (thought of as a 
function of the value of the underlying stock) in the other.
This extends a standard approach in optimal control/mathematical finance 
where pricing 
problems are solved via a Hamilton-Jacobi-Bellman equation which derives 
the value of an option for all values of the underlying stock 
simultaneously, even though only one of those values is relevant. 
Conversely, on a given date, options with several strikes may be traded.

The calculations given in \cite{alfonsi3} are most impressive, 
especially 
given the complicated and implicit nature of the results (see for example,
\cite[Theorem 3.2]{alfonsi2}; typically the diffusion 
coefficients are 
only specified up to the solution of one or more differential equations).
In contrast, our results on the inverse problem are often fully 
explicit. We find progress easier since the $g$-convex method does 
not require the solution of an intermediate differential equation. 
Furthermore,
the main duality (\cite[Theorem 3.1]{alfonsi2}) is nothing 
but the 
observation that the $g$-subdifferentials (respectively $\xstar$ and 
$\thstar$) are inverses of each other, which is immediate 
in the generalised convexity framework developed in this
article. Coupled with the identity,
$g_x(\xstar(\theta),\theta)=\psi'(\xstar(\theta))$ 
the main duality result of \cite{alfonsi3}
follows after algebraic manipulations, at least in
the smooth case 
where $\xstar$ and $\thstar$ are strictly increasing and differentiable, 
which is the setting of \cite{alfonsi3} and \cite{alfonsi2}.

\appendix
\section{Proofs}

\subsection{The Forward and Inverse Problems}

\noindent{\bf Proof of Lemma~\ref{l:coincide}}
Clearly $V \geq \hat{V}$, since the supremum over first hitting times 
must be
less than or equal to the supremum over all stopping times.

Conversely, by (\ref{hatv}), $\varphi(x) \geq
\frac{G(x,\theta)}{\hat{V}(\theta)}$. Moreover, (\ref{eq:differential}) 
implies that
$e^{- \rho t} \varphi(X_t)$ is a non-negative local martingale and
hence a supermartingale. Thus for stopping times $\tau$ we have
\[1 \geq \E_0[e^{-\rho \tau} \varphi(X_\tau)] \geq \E_0[e^{-\rho \tau}
G(X_{\tau},\theta) / \hat{V}(\theta)] \]
and hence $\hat{V}(\theta) \geq \sup_{\tau} E_0 [e^{-\rho \tau}
G(x_{\tau},\theta)]$.

\subsection{u-convexity}

The methods and many of the results we use in this section
are to be found in \cite{rachev} and \cite{carlier}.

\begin{lemma}
For every function $f: D_z \rightarrow \R$, $(f^u)^u$ is the largest 
$u$-convex minorant of $f$.
\end{lemma}

\begin{proof}

Note that $f^* \equiv -\infty$ is always a $u$-convex minorant and by 
Definition \ref{def:uconvex1} a function defined as the supremum of a 
collection of $u$-convex functions is $u$-convex.

Hence we can define $f^m$, the largest $u$-convex minorant of $f$ by
\[f^m(z)= \sup_\zeta \{ \zeta(x) ; \zeta \leq f, \zeta \in U(D_z) \}, \]
where $U(D_z)$ is the set of $u$-convex functions on $D_z$.
Then by the Fenchel Inequality, $f \geq (f^u)^u$ and since $(f^u)^u$ is 
$u$-convex, it follows that $f^m \geq (f^u)^u$. 

Since $f^m$ is $u$-convex there exists $S' \subset D_z \times \R$ such 
that 
\[f^m(y)=\sup_{(z,a)\in S'} [u(y,z)+a]\]
For fixed $(z,a) \in S'$, $l(y)=u(y,z)+a$ is $u$-convex and it is 
easy to check that $(l^u)^u(y) = l(y)$.

We note that if $f \leq \hat{f}$, then $f^u \geq \hat{f}^u$. Since 
$u(y,z)+a 
\leq f^m(y) \leq f(y)$, we therefore have $u(y,z)+a \leq 
(f^u)^u(y)$. Taking the supremum over pairs $(z,a) \in S'$, we get 
$f^m \leq (f^u)^u$.
\end{proof}

\noindent{\bf Proof of Lemma~\ref{def:uconvex2}}
If $f$ is $u$-convex then $f$ is its greatest $u$-convex minorant, so 
$f=(f^u)^u$. On the other hand if $(f^u)^u = f$ then since $(f^u)^u$ 
is $u$-convex, so is $f$.
%\end{proof}

\noindent{\bf Proof of Lemma~\ref{lem:monsubdiff}}
Let $y<\hat{y}$. For any $z \in \partial^u f(y)$ and any 
$\hat{z} \in \partial^u f(\hat{y})$, the following pair 
of inequalities 
follow 
from $u$-convexity (see especially the second part of 
Definition~\ref{def:subdifferential}):
\begin{eqnarray*}
f(y)-f(\hat{y}) & \geq & u(y,\hat{z}) - u(\hat{y},\hat{z}),  \\
f(\hat{y})-f(y) & \geq & u(\hat{y},z) - u(y,z).
\end{eqnarray*}
Adding and re-arranging we get,
\[u(y,\hat{z}) - u(\hat{y},\hat{z}) + u(\hat{y},z) - u(y,z) \leq 0, \]
and hence,
\begin{equation} \label{eq:increasingineq}
\int_y^{\hat{y}} [u_y(v,{z})-u_y(v,\hat{z})] dv \leq 0.
\end{equation}
Since (by Assumption~\ref{ass}(b)) $u_y(v,w)$ is strictly increasing in 
the second argument
it follows that $\hat{z} \geq z$.

\noindent{\bf Proof of Proposition \ref{p:prop1}}
Suppose $f$ is $u$-convex and $u$-subdifferentiable. Let $y$ be a point 
of differentiability of $f$ and $z \in \partial^u f(y)$. Then for $h>0$, 
we 
have both $f(y+h) - f(y) = f'(y) h + o(h)$ and
\[ f(\bar{y})- f(y) \geq u(\bar{y},z)-u(y,z) \hspace{20mm} \forall 
\bar{y} \in D_y. \]
From the definition of $u$-convexity $f(y) = u(y,z)-f^u(z)$ and $f(y+h) 
\geq u(y+h,z)-f^u(z)$, so taking $\bar{y}=y+h$ and putting all this 
together,
\[ f(y+h)-f(y) \geq u(y+h,z)-u(y,z) = u_y(y,z) + o(h). \]

Dividing by $h>0$ on both sides and letting $h \rightarrow 0$ we get
 \[ f '(y) \geq u_y(y,z) .\] 
If instead we take $h<0$ above then we 
get the reverse inequality and hence 
\begin{equation}
f '(y)=u_y(y,z).
\end{equation}

By the Spence-Mirrlees condition $u_y(y,z)$ is injective with 
respect to $z$. Hence, the last 
equation determines $z$ uniquely; thus whenever $f$ is differentiable 
the sub-differential $\partial^u f(y)$ is a singleton set.
We set $\partial^u f(y)= \{\zstar(y)\}$. 
Then, since $\zstar(y) \in \partial^u f(y)$ we have
$f(y)+ f^{u}(\zstar(y)) = u(y, \zstar(y))$ as required.

To prove the converse statement, 
suppose that for any point of differentiability $y$ of $f$, 
\[ f'(y)=u_y(y,\zstar(y))\] 
where $\zstar$ non-decreasing.

Define 
\[ \zeta(y)=\sup_v [f(v)+u(y,\zstar(v))-u(v,\zstar(v))] .\]

Then $\zeta$ is $u$-convex: to see this define  
\[S=\{(z,a); \exists w \in D_y \ | \ z=\zstar(w), \ 
a=f(w)-u(w,\zstar(w))\} \]
and then 
\[\zeta(y)=\sup_{(\theta,a) \in S} [u(y,\theta)+a]. \]
%from which $u$-convex follows by Definition \ref{def:uconvex1}.

Clearly $\zeta \geq f$. We want to show that $\zeta = f$.
By the assumptions on $f$ we have 
\[ f(y)= \int_{v}^y u_y(w,\zstar(w))dw + f(v). \]
Hence 
%\begin{equation}
\[ f(y)-(f(v)+u(y,\zstar(v))-u(v,\zstar(v)))= \int_v^y
[u_y(w,\zstar(w))-u_y(w,\zstar(w))] dw \geq 0.
\] %\end{equation}
Thus $\zeta =f$ and so $f$ is $u$-convex.
%\end{proof} 
 
\vspace{5mm}

The following Corollary is immediate from \ref{p:prop1}.

\begin{corollary} Suppose $f$ is differentiable at $y$. Then $f$ is 
twice-differentiable at $y$ if and only if $\zstar$ is differentiable 
at $y$.
\end{corollary}

%\begin{proof}
%$u(x,\theta)$ is twice continuously differentiable in $x$ and $\theta$ 
%and $\phi ' (x)=u_x(x,\thstar(x))$.
%\end{proof}

\noindent{\bf Proof of Proposition \ref{p:prop2}}

Suppose $f$ is $u$-subdifferentiable in a neighbourhood of $y$. Then 
for small enough $\epsilon$,
\[f(y+\epsilon)-f(y) \geq f(y+\e,z^*(y))- f(y,z^*(y)) \]
and $\lim_{\epsilon \downarrow 0} \{ f(y+\epsilon)-f(y) \}/\epsilon 
\geq f_y(y, z^*(y))$.

For the reverse inequality, if $ z^*$ is continuous at $y$ then for 
$\e$ small enough 
that $z^*(y + \e)<z^*(y)+\delta$ we have
\[ f(y+\epsilon)-f(y) \leq f(y+\e,z^*(y+\e))- f(y,z^*(y+\e)) \leq
f(y+\e,z^*(y)+\delta) - f(y,z^*(y)+\delta) \]
and $\lim_{\epsilon \downarrow 0} \{ f(y+\epsilon)-f(y) \}/\epsilon 
\leq \lim_{\delta \downarrow 0} f_y(y, z^*(y)+\delta) = f_y(y, 
z^*(y))$. 

Inequalities for the left-derivative follow similarly, and then 
$f'(y) = u_y(y,z^*(y))$ which is continuous.

Conversely, if $\partial^u f$ is multi-valued at $y$ so that $z^*$ is 
discontinuous at $y$, then
\[ \lim_{\epsilon \downarrow 0} \{ f(y+\epsilon)-f(y) \}/\epsilon 
\geq f_y(y, z^*(y)+) > f_y(y, z^*(y)-) \geq \lim_{\epsilon \downarrow 
0} \{ f(y)-f(y-\e) \}/\epsilon \]
where the strict middle inequality follows immediately from 
Assumption~\ref{ass}.

\bibliography{biblio}
\end{document}